\documentclass[11pt]{article}

\usepackage{fullpage}
\usepackage{parskip}
\usepackage{tikz} 
\usepackage{hyperref}
\usepackage{eqnarray,amsthm, amssymb, amsmath,verbatim,epsfig}
\usepackage{mathrsfs}
\usepackage[margin=1in]{geometry}
\usepackage{enumerate}
\usepackage{graphicx}
\usepackage{epstopdf}
\usepackage{multicol}
\usepackage{subfigure}

\newtheorem{thm}{Theorem}
\newtheorem{lem}{Lemma}

\newcommand{\tref}[1]{Theorem \ref{thm:#1}}
\newcommand{\lref}[1]{Lemma \ref{lem:#1}}
\newcommand{\fref}[1]{Figure \ref{fig:#1}}
\newcommand{\taref}[1]{Table \ref{table:#1}}
\newcommand{\cref}[1]{Conjecture \ref{conjecture:#1}}

\newcommand{\des}{\mathrm{des}}
\newcommand{\asc}{\mathrm{asc}}
\newcommand{\red}{\mathrm{red}}

\newcommand{\sg}{\sigma}
\newcommand{\la}{\lambda}
\newcommand{\ga}{\gamma}

\newcommand{\Sn}[1]{\mathcal{S}_{#1}}
\newcommand{\Dn}{\mathcal{D}_n}
\newcommand{\Snn}{\mathcal{S}_n}
\newcommand{\Snk}{\mathcal{S}_k}
\newcommand{\Des}{\mathrm{Des}}

\newcommand{\filll}[3]{\node at (#1,#2) {$#3$}}
\newcommand{\fillll}[3]{\node at (#1-.5,#2-.5) {$#3$}}
\newcommand{\filllll}[2]{\node at (#1-.5,#2-.5) {$#2$}}
\newcommand{\fillcross}[2]{\draw[thick] (#1-1,#2-1)--(#1,#2);\draw[thick] (#1-1,#2)--(#1,#2-1)}
\newcommand{\fillgcross}[2]{\draw[thick,green!50!black] (#1-1,#2-1)--(#1,#2);\draw[thick,green!50!black] (#1-1,#2)--(#1,#2-1)}
\newcommand{\thn}[1]{$#1^\textnormal{th}$}

\newcommand{\BBxt}[2]{{B}_{#1}^{\underline{#2}}}
\newcommand{\AAA}[3]{{A}_{#1}^{#2}(#3)}
\newcommand{\AAxyt}[2]{{A}_{#1}^{\underline{#2}}(t,y,x)}
\newcommand{\AAAover}[2]{{A}_{#1}^{\underline{#2}}}

\title{Counting Consecutive Pattern Matches in $\mathcal{S}_n(132)$ and $\mathcal{S}_n(123)$}

\author{
Ran Pan \\
\small Department of Mathematics\\[-0.8ex]
\small University of California, San Diego\\[-0.8ex]
\small La Jolla, CA 92093-0112. USA\\[-0.8ex]
\small \texttt{ran.pan.math@gmail.com}
\and
Dun Qiu \\
\small Department of Mathematics\\[-0.8ex]
\small University of California, San Diego\\[-0.8ex]
\small La Jolla, CA 92093-0112. USA\\[-0.8ex]
\small \texttt{duqiu@ucsd.edu}
\and
Jeffrey Remmel \\
\small Department of Mathematics\\[-0.8ex]
\small University of California, San Diego\\[-0.8ex]
\small La Jolla, CA 92093-0112. USA
}

\date{\today}

\begin{document}
\maketitle
\begin{abstract}
\noindent In this paper, we study the distribution of consecutive patterns in the set of 123-avoiding permutations and the set of 132-avoiding permutations, 
that is, in $\mathcal{S}_n(123)$ and $\mathcal{S}_n(132)$. 
We first study the distribution of consecutive pattern $\gamma$-matches in $\mathcal{S}_n(123)$ and $\mathcal{S}_n(132)$ for each length 3 consecutive pattern $\gamma$. Then we extend our methods to study the joint distributions of multiple consecutive patterns. Some more general cases are discussed in this paper as well.

\end{abstract}

\section{Introduction}

Let $\Snn$ denote the set of the permutations of length $n$. In this paper, we shall use one-line notation for permutations. For example, $\sg=145623\in \Sn6$ is a permutation of length $6$.

For a permutation $\sg = \sg_1\ldots \sg_n\in \Snn$, we let $\Des(\sg) :=
\{i: \sg_i>\sg_{i+1}\}$ be the set of \emph{descent positions}. We say that $i$ is a \emph{descent} of $\sg$ if $i\in \Des(\sg)$, and we let $\des(\sg) := |\Des(\sg)|$ be the number of {descents}. 
Similarly, we say that $i\in\{1,\ldots,n-1\}$ is an \emph{ascent} of $\sg$ if $i\notin \Des(\sg)$, and we let $\text{asc}(\sg)$  denote the number of ascents of $\sg$. Clearly, $\asc(\sg)+\des(\sg)=n-1$.

Given a sequence of $n$ distinct positive integers $w = w_1 \cdots w_n$, we let $\red(w)$ denote the permutation in $\Snn$ obtained from $w$ by replacing the \thn{i} smallest letter in $w$ by $i$. For example,
$\red(51)=21$ and $\red(2638)=1324.$

Given a permutation $\sg=\sg_1\ldots\sg_n \in \Snn$,
we say that a permutation $\la=\la_1\cdots\la_k \in \Snk$ \emph{occurs}
in $\sg$
if and only if there exist integers $1\leq i_1<i_2<\ldots<i_k\leq n$ such that
\begin{equation*}
\red(\sg_{i_1}\sg_{i_2}\ldots\sg_{i_k})=\la,
\end{equation*}
and $\sg$ \emph{avoids} $\la$ if $\la$ does not occur in $\sigma$. In the theory of permutation patterns, such $\la$ is called a \emph{classical
pattern}. We let $\Snn(\la)$ denote the set of permutations in $\Snn$ avoiding $\la$.

Similarly, 
we say that a permutation $\ga=\ga_1\cdots\ga_k \in \Snk$ \emph{occurs consecutively}
in $\sg\in\Snn$
if and only if there exists an integer $i$ such that
\begin{equation*}
\red(\sg_{i}\sg_{i+1}\ldots\sg_{i+k-1})=\ga,
\end{equation*}
in which case we say that $\sg$ has a consecutive $\ga$-match at the \thn{i} position of $\sg$, and $\sg$ \emph{avoids} $\ga$ \emph{consecutively} if $\ga$ does not occur consecutively in $\sigma$. Such $\ga$ is called a \emph{consecutive pattern}. We let $\ga\text{-mch}(\sigma)$ denote the number of consecutive $\ga$-matches in $\sigma$. Clearly, a descent is equivalent to a consecutive $21$-match. 

In short, a consecutive pattern requires the indices of  the subsequence in the permutation to be adjacent, while a classical pattern does not have this restriction.  
It is a fact that $\la\text{-mch}(\sigma)=0$ if $\sg\in\Snn(\la)$, while the converse is not always true which can be seen from the following example: let $\la = 132\in\Sn{3}$ and $\sigma = 23541\in\Sn{5}$, then $\la\text{-mch}(\sigma)=0$, but $\sg\notin\Snn(\la)$ since $\red(\sg_1\sg_3\sg_4)=\red(254)=\la$.

Both classical permutation patterns and consecutive permutation patterns are popular topics in enumerative combinatorics. One well-known result is that for any $\la \in \Sn3$, the number of permutations in $\Snn$ avoiding the classical pattern $\la$ is counted by $n^{\text{th}}$ Catalan number, that is,
\begin{equation*}
|\Sn{n}(\la)|=C_n=\frac{1}{n+1}\binom{2n}{n}, ~~\forall \la \in \Sn3.
\end{equation*}
More generally, combinatorists are interested in finding (exponential) generating functions such as
\begin{equation*}
F_{\la}{(t)}:=\sum_{n\geq 0}|\Sn{n}(\la)|t^n
\end{equation*}
and
\begin{equation*}
G_{\la}{(t,x)}:=\sum_{n\geq 0}\frac{t^n}{n!}\sum_{\alpha \in \Snn}x^{\la\text{-mch}(\alpha)}.
\end{equation*}
For two permutation patterns $\sg$ and $\la$, we say $\sg$ and $\la$ are \emph{Wilf equivalent} if
\begin{equation*}
F_{\sg}{(t)}\equiv F_{\la}{(t)}.
\end{equation*}
Similarly, we say $\sg$ and $\la$ are \emph{c-Wilf equivalent} if
\begin{equation*}
G_{\sg}{(t, 0)}\equiv G_{\la}{(t, 0)},
\end{equation*}
where `c' stands for `consecutive'. We say $\sg$ and $\la$ are \emph{strongly c-Wilf equivalent} if 
\begin{equation*}
G_{\sg}{(t, x)}\equiv G_{\la}{(t, x)}.
\end{equation*}

Although both classical patterns and consecutive patterns have been studied separately for a long time (e.g. \cite{Eli,Kit}), there is not much research about the enumeration of consecutive pattern matches in $\Snn(\la)$. Very recently, Bukata, Kulwicki, Lewandowski, Pudwell, Roth and Wheeland studied the distribution of valleys and peaks in $\Snn(123)$ in \cite{BKLPRW}. In this paper, we shall study a more general situation. Our focus is to enumerate  consecutive pattern matches in $\Snn(\la)$ for all $\la\in \Sn{3}$ and keep tracking the number of descents meanwhile. In other words, we want to study the following generating function:
\begin{equation}\label{eqn: single_gf}
\AAAover{\la}{\gamma}(t,y,x):=\sum_{n\geq 0}t^n\sum_{\sg \in \Snn(\la)}y^{\des(\sg)}x^{\gamma\text{-mch}(\sg)}.
\end{equation}
where $\lambda \in \Sn3$. Furthermore, we can also keep track of multiple consecutive patterns in $\Sn{\la}$ and define the following generating function,
\begin{equation}
\AAAover{\lambda}{\Gamma}{(t,y,x_1,\ldots,x_s)}:=\sum_{n\geq 0}t^n\sum_{\sg\in\Snn(\lambda)}y^{\des(\sg)}\prod_{i=1}^s x_i^{\ga_i\text{-mch}(\sg)},
\end{equation}
where $\Gamma=\{\gamma_1,\ldots,\gamma_s\}$ is a set of consecutive patterns.

As mentioned above, for any $\la \in \Sn3$, $|\Snn(\la)|$ is counted by the Catalan number $C_n$ which also counts the number of $(n,n)$-Dyck paths. We will mainly use bijections between $\Snn(\la)$ and $(n,n)$-Dyck paths to find formulas for the functions $A_\la^{\underline{\Gamma}}{(t,y,x_1,\ldots,x_s)}$ and $\AAAover{\la}{\gamma}(t,y,x)$. 

The outline of this paper is as follows. In Section 2, we will study some symmetries about the function $\AAAover{\la}{\gamma}(t,y,x)$ in terms of $\gamma$. We will also review the bijection between $\Snn(132)$ and $(n,n)$-Dyck paths introduced by Krattenthaler in \cite{Kr}, the bijection  between $\Snn(123)$ and $(n,n)$-Dyck paths introduced by  Deutsch and Elizalde in \cite{DE}, and three  recursions for the generating function of Dyck paths. In Section 3, we shall compute generating functions in the form of $\AAAover{\la}{\ga}{(t,y,x)}$ for $\la, \ga \in \Sn{3}$. In Section 4, we shall study the generating functions of $\Snn(123)$ and $\Snn(132)$ tracking multiple consecutive patterns. In particular, we compute $\AAA{123}{\underline{\{132,231,321\}}}{t,y,x_1,x_2,x_3}$
and
$\AAA{132}{\underline{\{123,213,231,321\}}}{t,y,x_1,x_2,x_3,x_4}$.
In Section 5 and Section 6, we will study generating functions of the form
\begin{equation}
\BBxt{\la}{\ga}(t,x):=\sum_{n\geq 0}t^n\sum_{\sg\in\Snn(\lambda)}x^{\ga\text{-mch}(\sg)},
\end{equation}
where $\la=123$ or $132$, and $\ga$ is of some special shapes.

\section{Preliminaries}
\subsection{Symmetries about the function $\AAAover{\la}{\gamma}(t,y,x)$}
We shall begin with studying some symmetries about consecutive pattern distributions in $\Snn(132)$ and $\Snn(123)$.

Given a permutation $\sg = \sg_1 \sg_2 \ldots \sg_n \in \Snn$, we let $\sg^r$ be the \emph{reverse} of $\sg$ defined by $\sg^r = \sg_n \ldots \sg_2 \sg_1$, and $\sg^c$ be the \emph{complement} of $\sg$ defined by $\sg^c = (n+1 -\sg_1) (n+1-\sg_2) \ldots (n+1 -\sg_n)$. Further, we let $\sg^{rc}=(\sg^r)^c$ be the \emph{reverse-complement} of $\sg$. For example, for $\sg=15324$, we have $\sg^r=42351,\ \sg^c=51342,\mbox{ and } \sg^{rc}=24315$.

The three actions give several symmetries about consecutive pattern distributions in $\Snn(\la)$, and we have the following lemma:
\begin{lem}\label{lem:rcrc}
For any permutations $\la$ and $\ga$, we have
\begin{eqnarray}
\label{rc}
\AAAover{\la^{rc}}{\gamma^{rc}}(t,y,x) &=& \AAAover{\la}{\gamma}(t,y,x), \\
\label{r}
\AAAover{\la^r}{\gamma^r}(t,y,x) &=& 1+ \frac{\AAAover{\la}{\gamma}(ty,\frac{1}{y},x)-1}{y}, \mbox{ and}\\
\label{c}
\AAAover{\la^c}{\gamma^c}(t,y,x) &=& 1+ \frac{\AAAover{\la}{\gamma}(ty,\frac{1}{y},x)-1}{y}.
\end{eqnarray}
\end{lem}
\begin{proof}
The action reverse-complement sends a permutation $\sg\in\Snn(\la)$ to $\sg^{rc}\in\Snn(\la^{rc})$, while each consecutive $\ga$-match in $\sg$ is sent to a consecutive $\ga^{rc}$-match in $\sg^{rc}$, and each descent in $\sg$ is sent to a descent in $\sg^{rc}$,  thus $\ga^{rc}\text{-mch}(\sg^{rc})=\ga\text{-mch}(\sg)$, $\des(\sg^{rc})=\des(\sg)$, and we have equation (\ref{rc}).

The action reverse sends a permutation $\sg\in\Snn(\la)$ to a permutation $\sg^r\in\Snn(\la^r)$, while each consecutive $\ga$-match in $\sg$ is sent to a consecutive $\ga^r$-match in $\sg^r$, and each descent in $\sg$ is sent to an ascent in $\sg^{r}$,  thus $\ga^{r}\text{-mch}(\sg^{r})=\ga\text{-mch}(\sg)$, $\des(\sg^{r})=\asc(\sg)=n-1-\des(\sg)$. Then
\begin{eqnarray*}
\AAAover{\la^r}{\gamma^r}(t,y,x) &=& 1+\sum_{n\geq 0}t^n\sum_{\sg \in \Snn(\la^r)}y^{\des(\sg)}x^{\gamma^r\text{-mch}(\sg)}\\
&=& 1+\sum_{n\geq 0}t^n\sum_{\sg \in \Snn(\la)}y^{n-1-\des(\sg)}x^{\gamma\text{-mch}(\sg)}\\
&=& 1 + \frac{\sum_{n\geq 0}(ty)^n\sum_{\sg \in \Snn(\la)}y^{-\des(\sg)}x^{\gamma\text{-mch}(\sg)}}{y},
\end{eqnarray*}
which proves equation (\ref{r}). Equation (\ref{c}) follows in a similar way.
\end{proof}
Since $123=321^r$ and $132=231^r=312^c=213^{rc}$, we shall only consider the case when $\la=123$ or $132$ to compute the generating functions of the form $\AAAover{\la}{\gamma}(t,y,x)$ when $\la$ is of length 3.

Note that for any permutation $\sg\in\Snn(\la)$, the number of consecutive $\la$-matches of $\sg$ is always 0, therefore the functions $\AAAover{123}{123}(t,y,x)$ and $\AAAover{132}{132}(t,y,x)$ are trivial.

Since $123^{rc}=123$, we have 
\begin{lem}\label{lem:123eq}
	For any pattern $\ga$,
	\begin{equation}
	\AAAover{123}{\ga}(t,y,x)=\AAAover{123}{\ga^{rc}}(t,y,x).
	\end{equation}
\end{lem}
Setting $\ga=132$ and $\ga=231$, it follows that
\begin{eqnarray}
\label{sym1}
\AAAover{123}{132}(t,y,x)&=&\AAAover{123}{213}(t,y,x),\\
\label{sym2}
\AAAover{123}{231}(t,y,x)&=&\AAAover{123}{312}(t,y,x).
\end{eqnarray}

On the other hand, we shall use a bijective map $\phi_n:\Snn(312) 
\rightarrow \Snn(213)$ in \cite{QR2} to study symmetries about the function $\AAAover{132}{\ga}(t,y,x)$. 

Given a permutation $\sg=\sg_1\cdots\sg_n\in\Snn(312)$, the permutation $\phi_n(\sg)$ is defined recursively as follows: if $n=1$, then $\phi_n(\sg)=\sg$. 
If $n>1$, then we suppose that $\sg_r=1$. Since $\sg$ is 312-avoiding, it must be the case that all the numbers $\sg_{r+1},\ldots,\sg_n$ are bigger than all the numbers $\sg_1,\ldots,\sg_{r-1}$. Thus, $\sg_1\cdots\sg_{r-1}$ is a 312-avoiding permutation of the set $\{2,\ldots,r\}$ and $\sg_{r+1}\cdots\sg_n$ is a 312-avoiding permutation of the set $\{r+1,\ldots,n\}$. Setting $\overline{\sg}_1\cdots\overline{\sg}_{r-1}=\phi_{r-1}((\sg_1-1)\cdots(\sg_{r-1}-1))$ and $\overline{\sg}_{r+1}\cdots\overline{\sg}_{n}=\phi_{n-r}((\sg_{r+1}-r)\cdots(\sg_{n}-r))$, we have
\begin{equation*}
\phi_n(\sg)=(\overline{\sg}_1+n-r+1)\cdots (\overline{\sg}_{r-1}+n-r+1)\ 1\ (\overline{\sg}_{r+1}+1)\cdots(\overline{\sg}_{n}+1).
\end{equation*}

\fref{phin} shows an example of this map that $\phi_5(32415)=53412$. For any permutation $\sg\in\Snn(312)$, it is proved in (\cite{QR2}, Theorem 1) that $\phi_n$ does not change the descent positions of $\sg$, i.e.\ $\Des(\sg)=\Des(\phi_n(\sg))$. As a consequence, $k\cdots 21\text{-mch}(\sg)=k\cdots 21\text{-mch}(\phi_n(\sg))$  since the consecutive pattern $k\cdots 21$ can be seen as $k-1$ consecutive descents, and we have
\begin{equation}\label{312sym}
\AAAover{312}{k\cdots 21}(t,y,x) = \AAAover{213}{k\cdots 21}(t,y,x).
\end{equation}

\begin{figure}[ht]
	\centering
	\vspace{-1mm}
	\begin{tikzpicture}[scale =.5]
	\draw[help lines] (0,0) grid (5,5);
	\draw[very thick] (0,1) rectangle (3,4);
	\draw[very thick] (4,4) rectangle (5,5);
	\fillll{1}{3}{3};\fillll{2}{2}{2};
	\fillll{3}{4}{4};\fillll{4}{1}{1};
	\fillll{5}{5}{5};
	\draw[thick] (3.5,.5) circle [radius=0.5];
	\end{tikzpicture}
	\begin{tikzpicture}[scale =.5]
	\draw[help lines] (0,0) grid (5,5);
	\draw[very thick] (0,2) rectangle (3,5);
	\draw[very thick] (0,3) rectangle (1,4);
	\draw[very thick] (2,4) rectangle (3,5);
	\draw[very thick] (4,1) rectangle (5,2);
	\fillll{2}{3}{3};\fillll{5}{2}{2};
	\fillll{1}{4}{4};\fillll{4}{1}{1};
	\fillll{3}{5}{5};
	\draw[thick] (3.5,.5) circle [radius=0.5];
	\draw[thick] (1.5,2.5) circle [radius=0.5];
	\draw (-1,2.5) node {$\Longrightarrow$};
	\end{tikzpicture}	
	\begin{tikzpicture}[scale =.5]
	\draw[very thick] (0,2) rectangle (3,5);
	\draw[very thick] (0,4) rectangle (1,5);
	\draw[very thick] (2,3) rectangle (3,4);
	\draw[very thick] (4,1) rectangle (5,2);
	\draw[help lines] (0,0) grid (5,5);
	\fillll{2}{3}{3};\fillll{5}{2}{2};
	\fillll{3}{4}{4};\fillll{4}{1}{1};
	\fillll{1}{5}{5};
	\draw[thick] (1.5,2.5) circle [radius=0.5];
	\draw (-1,2.5) node {$\Longrightarrow$};
	\end{tikzpicture}
	\caption{$\phi_5(\sg)=53412\in\Sn{5}(213)$ for $\sg=32415\in\Sn{5}(312)$.}
	\label{fig:phin}
\end{figure}
From the recursive definition of the bijection $\phi_n$, the number of consecutive $1k\cdots32$-matches is not changed in each recursive step, and we have
\begin{equation}\label{312sym2}
\AAAover{312}{1k\cdots32}(t,y,x)=\AAAover{213}{1k\cdots32}(t,y,x),
\end{equation}
which implies the following lemma.
\begin{lem}\label{lem:1321}For any integer $k\geq 2$, we have
	\begin{eqnarray}
	\AAAover{132}{12\cdots k}(t,y,x) &=& 1+\frac{\AAAover{132}{k\cdots 21}(ty,\frac{1}{y},x)-1}{y},\\
	\AAAover{132}{k12\cdots(k-1)}(t,y,x)&=&1+\frac{\AAAover{132}{(k-1)\cdots 21k}(ty,\frac{1}{y},x)-1}{y}.
	\end{eqnarray}
\end{lem}
\begin{proof}
Since $312=132^c$ and $213=132^{rc}$, it follows from \lref{rcrc} and equations (\ref{312sym}) and (\ref{312sym2}) that 
\begin{eqnarray*}
\AAAover{132}{12\cdots k}(t,y,x) &=& 1+ \frac{\AAAover{312}{k\cdots 2 1}(ty,\frac{1}{y},x)-1}{y} = 1+ \frac{\AAAover{213}{k\cdots 2 1}(ty,\frac{1}{y},x)-1}{y}=1+\frac{\AAAover{132}{k\cdots  21}(ty,\frac{1}{y},x)-1}{y},\\
\AAAover{132}{k12\cdots(k-1)}(t,y,x) &=& 1+ \frac{\AAAover{312}{1k\cdots32}(ty,\frac{1}{y},x)-1}{y} = 1+ \frac{\AAAover{213}{1k\cdots32}(ty,\frac{1}{y},x)-1}{y} \\&=& 1+\frac{\AAAover{132}{(k-1)\cdots 21k}(ty,\frac{1}{y},x)-1}{y}.\qedhere
\end{eqnarray*}
\end{proof}
When $k=3$, we have
\begin{eqnarray}
\label{sym3}
\AAAover{132}{123}(t,y,x)&=&1+\frac{\AAAover{132}{321}(ty,\frac{1}{y},x)-1}{y},\\
\label{sym4}
\AAAover{132}{312}(t,y,x)&=&1+\frac{\AAAover{132}{213}(ty,\frac{1}{y},x)-1}{y}.
\end{eqnarray}
Setting $y=1$ in the equations above, we have
$\AAAover{132}{123}(t,1,x)=\AAAover{132}{321}(t,1,x)$ and  
$\AAAover{132}{312}(t,1,x)=\AAAover{132}{213}(t,1,x).$

According to the symmetries given in equations (\ref{sym1}), (\ref{sym2}), (\ref{sym3}) and (\ref{sym4}), to compute $\AAAover{\la}{\ga}(t,y,x)$ for all $\la, \ga \in\Sn3$, it suffices to compute the following six generating functions:
$\AAAover{123}{132}(t,y,x)$,
$\AAAover{123}{231}(t,y,x)$ and
$\AAAover{123}{321}(t,y,x)$ for $\Snn(123)$;
$\AAAover{132}{123}(t,y,x)$,
$\AAAover{132}{213}(t,y,x)$ and
$\AAAover{132}{231}(t,y,x)$ for $\Snn(132)$.

\subsection{Bijections from $\Snn(132)$ and $\Snn(123)$ to $(n,n)$-Dyck paths}

Given an $n\times n$ square, we will label the coordinates of the columns from left to right
and the coordinates of the rows from top to bottom with $0,1, \ldots, n$.  An $(n,n)$-\emph{Dyck path} is a path consisting of unit down-steps $D$ and unit right-steps $R$ which starts at $(0,0)$ and ends at $(n,n)$ and stays on or below the diagonal $y=x$ (these are ``down-right" Dyck paths). The set of $(n,n)$-Dyck paths is denoted by $\mathcal{D}_n$.

Given a Dyck path $P$, 
we let the first return of $P$, denoted by $\mathrm{ret}(P)$, be the smallest number $i>0$ such that $P$ goes through the point $(i,i)$.  For example, for 
$
P =DDRDDRRRDDRDRDRRDR
$
shown in  \fref{Dpath}, 
$\mathrm{ret}(P) =4$ since the leftmost point on the diagonal after $(0,0)$ that $P$ goes through is $(4,4)$. 
We refer to positions $(i,i)$ where $P$ goes through as \emph{return positions} of $P$. We shall also label the diagonals that go through corners of squares 
that are parallel to and below the main diagonal with $0, 1, 2, \ldots $ starting at the main diagonal, as shown in \fref{Dpath}. The \emph{peaks} of a path $P$ are the positions of consecutive $DR$ steps. 

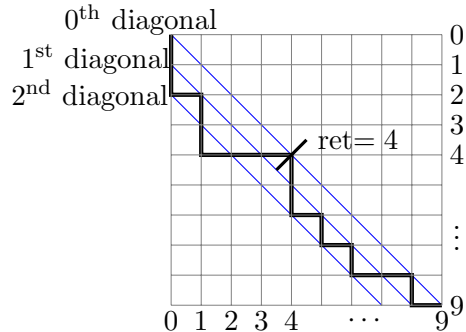
\begin{figure}[ht]
	\centering
	\vspace{-3mm}	
	\begin{tikzpicture}[scale =.4]
	\draw[blue] (0,9)--(9,0);
	\node[] at (-1,9.5) {$0^{\textnormal{th}}$ diagonal};
	\draw[blue] (0,8)--(8,0);
	\node[] at (-2.5,8.2) {$1^{\textnormal{st}}$ diagonal};
	\draw[blue] (0,7)--(7,0);
	\node[] at (-2.7,7) {$2^{\textnormal{nd}}$ diagonal};
	\path (-4,4.5);
	\draw[ultra thick] (0,9)--(0,7) -- (1,7)--(1,5)-- (4,5)--(4,3)--(5,3)--(5,2)--(6,2)--(6,1)--(8,1)--(8,0)--(9,0);
	\draw[help lines] (0,0) grid (9,9);	
	\draw[very thick] (3.5,4.5)--(4.5,5.5) node[right] {ret$=4$};
	\filll{0}{-.5}{0};\filll{1}{-.5}{1};\filll{2}{-.5}{2};\filll{3}{-.5}{3};\filll{4}{-.5}{4};
	\filll{6.5}{-.5}{\cdots};
	\filll{9}{-.5}{9};
	\filll{9.5}{0}{9};\filll{9.5}{9}{0};\filll{9.5}{8}{1};\filll{9.5}{7}{2};\filll{9.5}{6}{3};
	\filll{9.5}{5}{4};
	\filll{9.5}{2.5}{\vdots};
	\end{tikzpicture}
	\caption{an $(n,n)$-Dyck path.}
	\label{fig:Dpath}
	\vspace{-1mm}
\end{figure}

As we have mentioned, 
$|\Snn(132)|=|\Snn(123)|=|\mathcal{D}_n|=C_n$, where $C_n = \frac{1}{n+1}\binom{2n}{n}$ is the 
$n^{\mathrm{th}}$ Catalan number. Many bijections are known between these Catalan objects (see \cite{Stan}), and we shall apply the bijection of Krattenthaler \cite{Kr} between $\Snn(132)$ and $\mathcal{D}_n$ and the bijection of Deutsch and Elizalde \cite{DE} between $\Snn(123)$ and $\mathcal{D}_n$ in our computation of generating functions. The second and the last author of this paper also discussed the two bijections in \cite{QR1,QR2} with more details.

We shall first describe the bijection $\Phi$ of Krattenthaler \cite{Kr} between $\Snn(132)$ and $\mathcal{D}_n$. 
Given any permutation $\sg = \sg_1 \cdots \sg_n \in\Snn(132)$, we write it on an $n\times n$ table 
by placing $\sg_i$ in the $i^{\mathrm{th}}$ column counting from left to right and $\sg_i^{\mathrm{th}}$ row counting from bottom to top. Then for each $i$, we 
shade the cells to the north-east of the cell that contains $\sg_i$. $\Phi(\sg)$ 
is the path that goes along the south-west boundary of the shaded cells. For example, this 
process is pictured in \fref{SnDn1} in the case where  $\sg=867943251\in\Sn{9}(132)$. 
In this case, $\Phi(\sg)=  DDRDDRRRDDRDRDRRDR$.

\begin{figure}[ht]
	\centering
	\subfigure[the map $\Phi:\Snn(132)\rightarrow\mathcal{D}_n$.\label{fig:SnDn1}]{\begin{tikzpicture}[scale =.35]
		\path[fill,black!15!white] (0,7) -- (1,7)--(1,5)-- (4,5)--(4,3)--(5,3)--(5,2)--(6,2)--(6,1)--(8,1)--(8,0)--(9,0)--(9,9)--(0,9);
		\draw[help lines] (0,0) grid (9,9);
		\filllll{1}{8};\filllll{2}{6};
		\filllll{3}{7};\filllll{4}{9};
		\filllll{5}{4};\filllll{6}{3};
		\filllll{7}{2};\filllll{8}{5};
		\filllll{9}{1};
		\path (0,-.5);
		\end{tikzpicture}	
		\begin{tikzpicture}[scale =.35]
		\draw[help lines] (0,0) grid (9,9);
		\fillgcross{1}{8};\fillgcross{2}{6};
		\fillcross{3}{7};\fillcross{4}{9};
		\fillgcross{5}{4};\fillgcross{6}{3};
		\fillgcross{7}{2};\fillcross{8}{5};
		\fillgcross{9}{1};	
		\draw (-1.5,4.5) node {$\rightarrow$};
		\path (-2.5,4.5);
		\path (0,-.5);
		\draw[very thick] (0,9)--(0,7) -- (1,7)--(1,5)-- (4,5)--(4,3)--(5,3)--(5,2)--(6,2)--(6,1)--(8,1)--(8,0)--(9,0);
		\path (10,0);
		\end{tikzpicture}}\hspace*{5mm}
	\subfigure[the map $\Psi:\Snn(123)\rightarrow\mathcal{D}_n$.\label{fig:SnDn2}]{\begin{tikzpicture}[scale =.35]
		\path[fill,black!15!white] (0,7) -- (1,7)--(1,5)-- (4,5)--(4,3)--(5,3)--(5,2)--(6,2)--(6,1)--(8,1)--(8,0)--(9,0)--(9,9)--(0,9);
		\draw[help lines] (0,0) grid (9,9);
		\filllll{1}{8};\filllll{2}{6};
		\filllll{3}{9};\filllll{4}{7};
		\filllll{5}{4};\filllll{6}{3};
		\filllll{7}{2};\filllll{8}{5};
		\filllll{9}{1};
		\path (0,-.5);
		\end{tikzpicture}	
		\begin{tikzpicture}[scale =.35]
		\draw[help lines] (0,0) grid (9,9);
		\fillgcross{1}{8};\fillgcross{2}{6};
		\fillcross{3}{9};\fillcross{4}{7};
		\fillgcross{5}{4};\fillgcross{6}{3};
		\fillgcross{7}{2};\fillcross{8}{5};
		\fillgcross{9}{1};	
		\draw (-1.5,4.5) node {$\rightarrow$};
		\path (-2.5,4.5);
		\draw[very thick] (0,9)--(0,7) -- (1,7)--(1,5)-- (4,5)--(4,3)--(5,3)--(5,2)--(6,2)--(6,1)--(8,1)--(8,0)--(9,0);
		\path (0,-.5);
		\end{tikzpicture}}
	\vspace{-1mm}
	\caption{the maps $\Phi$ and $\Psi$.}
	\label{fig:132Dn}
	\vspace{-5mm}
\end{figure}
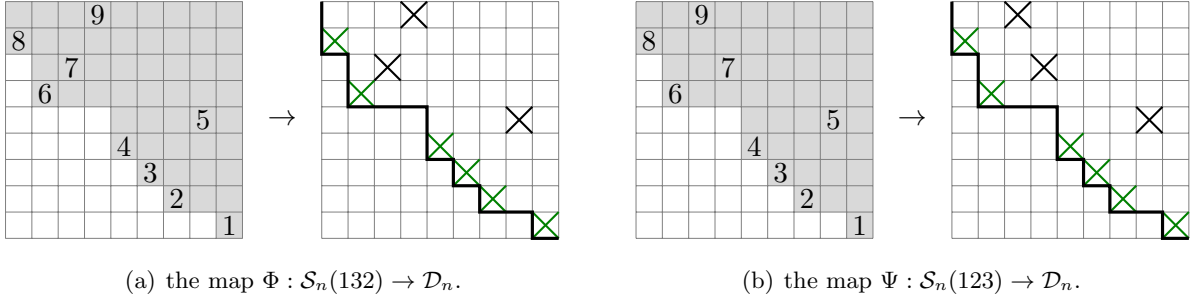

Given any permutation $\sg = \sg_1 \ldots \sg_n$, we say that $\sg_j$ is a \emph{left-to-right minimum} of $\sg$ if 
$\sg_i > \sg_j$ for all $i < j$. It is easy to see that the left-to-right minima 
of $\sg$ correspond to the \emph{peaks} of the path $\Phi(\sg)$. 
For example, for the permutation $\sg$ pictured in 
\fref{SnDn1}, there are $6$ left-to-right 
minima,  $\{8,6,4,3,2,1\}$.

The \emph{horizontal segments} (or \emph{segments}) of the path $\Phi(\sg)$ are the maximal consecutive 
sequences of $R$ steps in $\Phi(\sg)$. For example, in \fref{SnDn1}, the lengths of 
the horizontal segments, reading from top to bottom, are $1,3,1,1,2,1$, and $\{6,7,9\}$ is the set of numbers associated to the second horizontal segment of $\Phi(\sg)$. 

The map $\Phi$ is invertible since for each Dyck path $P$, the peaks of $P$ give the left-to-right minima of the $132$-avoiding permutation, and the remaining numbers are uniquely determined by the left-to-right minima. More details about  $\Phi$ can be found in \cite{Kr}. We have the following properties for $\Phi$ from \cite{QR1}.

\begin{lem}[\cite{QR1}, Lemma 3]\label{lem:p1}
	Let $P =\Phi(\sg)$ where $\sg\in\Snn(132)$. Then for each horizontal segment $H$ of $P$, the numbers associated 
	to $H$ form a consecutive increasing sequence in $\sg$, and the least 
	number of the sequence sits immediately above the first right-step of $H$. Hence 
	the descents in $\sg$ only occur between two different horizontal segments of $P$.  The number $n$ is in the column of last right-step before the first return. 
\end{lem}

The bijection $\Psi:\Snn(123) \rightarrow \mathcal{D}_n$  given 
by Deutsch and Elizalde \cite{DE} can be described in a similar way. 
Given any permutation $\sg \in\Snn(123)$, the path $\Psi(\sg)$ is constructed exactly in the same way as the map $\Phi$. \fref{SnDn2} shows an example of this map that $\Psi(\sg)=DDRDDRRRDDRDRDRRDR$ where $\sg=869743251\in\Sn{9}(123)$. 

The map $\Psi$ is invertible for the same reason that each $123$-avoiding permutation has a unique left-to-right minima set. More details about  $\Psi$ can be found in \cite{DE}. We have the following lemma from \cite{QR1}.

\begin{lem}[\cite{QR1}, Lemma 4]\label{lem:p2} Let $P =\Psi(\sg)$ where $\sg\in\Snn(123)$. Then for each horizontal segment $H$ of $P$, the least element of the 
set of numbers associated to $H$ sits directly above the first right-step of $H$,
and the remaining numbers of the set form a consecutive decreasing sequence in $\sg$. 
	
\end{lem}

\subsection{Recursions for Dyck paths}\label{sec:recursions}
We 
review three Dyck path recursions for the generating function
\begin{equation}
D(x)=\sum_{n\geq0}|\mathcal{D}_n|x^n=\sum_{n\geq0}C_n x^n,
\end{equation}
which
will help us to develop recursions to compute 
generating functions of the form $\AAAover{\lambda}{\gamma}{(t,y,x)}$.

{\bf Recursion 1.} The contribution of the empty path (the path of size 0) to $D(x)$ is $1$. If the size of the path is not 0, we can expand the last horizontal segment of the path. When the last horizontal segment is of length $k>0$, the path ends with steps $DR^k$. The total weight of the last horizontal segment is $x^k$ as it stretches over $k$ columns. Referring to \fref{Dyck1}, 
the steps before $DR^k$ can be decomposed into $k$ smaller Dyck path structures: one path after the last $D$ step on the $(i-1)^\textnormal{th}$ diagonal before the last $D$ step on the \thn{i} diagonal for $i=0,\ldots,k-1$, each with a weight of $D(x)$.

\begin{figure}[ht]
	\centering
	\vspace{-1mm}
	\subfigure[\label{fig:Dyck1a}]{\begin{tikzpicture}[scale =.3]
		\draw (0,3)--(0,2)--(1,2)--(-2,5)--(-2,3)--(0,3);
		\node at (-3,3.5) {\tiny $D(x)$};
		\node at (-.5,2.5) {\tiny $x$};
		\path (4,11);
		\end{tikzpicture}}
	\subfigure[\label{fig:Dyck1b}]{\begin{tikzpicture}[scale =.3]
		\draw (0,6)--(0,3)--(2,3)--(2,2)--(4,2)--(-2,8)--(-2,6)--(0,6)--(0,5)--(3,2);
		\node at (-1,5.5) {\tiny $D(x)$};
		\node at (1,2.5) {\tiny $D(x)$};
		\node at (3,1.5) {\tiny $x^2$};
		\path (4,11);
		\end{tikzpicture}}
	\subfigure[\label{fig:Dyck1c}]{\begin{tikzpicture}[scale =.3]
		\draw (0,9)--(0,6)--(2,6)--(2,3)--(4,3)--(4,2)--(7,2)--(0,9)--(-2,11)--(-2,9)--(0,9)--(0,8)--(6,2)--(5,2)--(2,5);
		\node at (-1,8.5) {\tiny $D(x)$};
		\node at (1,5.5) {\tiny $D(x)$};
		\node at (3,2.5) {\tiny $D(x)$};
		\node at (5.5,1.5) {\tiny $x^3$};
		\end{tikzpicture}}
	\caption{the first recursion for $D(x)$ when the last segment is of size $k=1,2,3$.}
	\label{fig:Dyck1}
\end{figure}
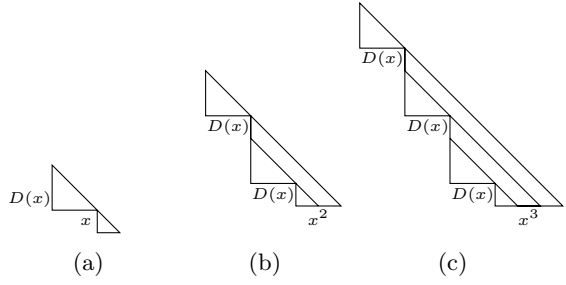
Thus the contribution of the case when the last horizontal segment is of size $k$ is $x^k D(x)^k$, and $D(x)$ satisfies the following recursive equation:
\begin{equation}
D(x)=1+\sum_{k>0}x^k D(x)^k.
\end{equation}

{\bf Recursion 2.} The contribution of the empty path to the generating function $D(x)$ is $1$. If the size of the path $P$ is not 0, we can get a recursion by breaking $P$ at the first return position into $2$ smaller Dyck paths: $P=DP_1RP_2$. Here the path $P_1$ locates between the first $D$ step and the last $R$ step before the first return, and $P_2$ locates after the first return, each with a weight of $D(x)$; the $R$ step before the first return has a weight of $x$ since it takes up one column. Thus we have
\begin{equation}
D(x)=1+xD(x)^2.
\end{equation}

{\bf Recursion 3.} The third Dyck path recursion can be seen as the combination of the first recursion and the second recursion. 
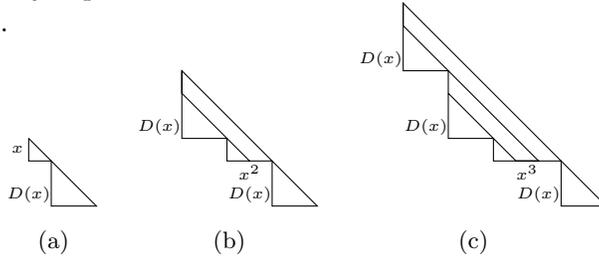
\begin{figure}[ht]
	\centering
	\vspace{-8mm}
	\subfigure[\label{fig:Dyck3a}]{\begin{tikzpicture}[scale =.3]
		\draw (0,3)--(0,2)--(1,2)--(1,0)--(3,0)--(0,3);
		\node at (0,.5) {\tiny $D(x)$};
		\node at (-.5,2.5) {\tiny $x$};
		\path (4,9);
		\end{tikzpicture}}
	\subfigure[\label{fig:Dyck3b}]{\begin{tikzpicture}[scale =.3]
		\draw (0,6)--(0,3)--(2,3)--(2,2)--(4,2)--(4,0)--(6,0)--(0,6)--(0,5)--(3,2);
		\node at (-1,3.5) {\tiny $D(x)$};
		\node at (3,.5) {\tiny $D(x)$};
		\node at (3,1.5) {\tiny $x^2$};
		\path (7,9);
		\end{tikzpicture}}
	\subfigure[\label{fig:Dyck3c}]{\begin{tikzpicture}[scale =.3]
		\draw (0,9)--(0,6)--(2,6)--(2,3)--(4,3)--(4,2)--(7,2)--(7,0)--(9,0)--(0,9)--(0,8)--(6,2)--(5,2)--(2,5);
		\node at (-1,6.5) {\tiny $D(x)$};
		\node at (1,3.5) {\tiny $D(x)$};
		\node at (6,.5) {\tiny $D(x)$};
		\node at (5.5,1.5) {\tiny $x^3$};
		\end{tikzpicture}}
	\vspace{-3mm}
	\caption{the third recursion for $D(x)$ when the last segment before first return is of size $k=1,2,3$.}
	\label{fig:Dyck3}
\end{figure}

The contribution of the empty path to the generating function $D(x)$ is $1$. If the size of the path $P$ is not 0, we shall expand the last horizontal segment before the first return. Suppose that the last horizontal segment before the first return is of length $k>0$, then $P$ ends with steps $D R^k$ before the first return, and the steps $D R^k$ has a weight of $x^k$. Referring to \fref{Dyck3}, the part of path after the first return forms a Dyck path with a weight of $D(x)$, and the part of path before the steps $D R^k$ can be decomposed into $k-1$ smaller Dyck path structures: 
one path after the last $D$ step on the $(i-1)^\textnormal{th}$ diagonal before the last $D$ step on the \thn{i} diagonal for $i=1,\ldots,k-1$, each with a weight of $D(x)$. The contribution of the case when the last horizontal segment is of size $k$ is $D(x)\cdot x^k D(x)^{k-1}$, and we have
\begin{equation}
D(x)=1+D(x)\sum_{k>0}x^kD(x)^{k-1}.
\end{equation}

\section{Computation of the generating function $\AAAover{\la}{\ga}{(t,y,x)}$}

In this section, we compute generating functions of the form
\begin{equation}
\AAAover{\la}{\ga}{(t,y,x)}:=\sum_{n\geq 0}t^n\sum_{\sg \in \Snn(\la)}y^{\des(\sg)}x^{\gamma\text{-mch}(\sg)}
\end{equation}
for $\la$ and $\ga$ in $\Sn{3}$. As discussed in Section 2.1, we only need to compute the following $6$ generating functions: 
$\AAAover{123}{132}(t,y,x)$,
$\AAAover{123}{231}(t,y,x)$,
$\AAAover{123}{321}(t,y,x)$,
$\AAAover{132}{123}(t,y,x)$,
$\AAAover{132}{213}(t,y,x)$ and
$\AAAover{132}{231}(t,y,x)$.

\subsection{The function $\AAxyt{123}{132}$}
Given any permutation $\sg\in\Snn(123)$, Deutsch and Elizalde's bijection and \lref{p2}\ implies that $\sg_i<\sg_{i+1}$ if and only if $\sg_i$ sits immediately above the first $R$ step of some horizontal segment of the corresponding Dyck path $\Psi(\sg)$, thus $\des(\sg)$ is equal to the total number of Dyck path patterns $RD$ and $RRR$ in  $\Psi(\sg)$.
Furthermore, $132\text{-mch}(\sg)$ is equal to the number of path pattern $DRRR$-matches (which looks like \begin{tikzpicture}[scale =.25]\draw[help lines] (3,1) grid (0,2);\draw[ultra thick] (0,2)--(0,1)--(3,1);\end{tikzpicture} on the $n\times n$ square) in $\Psi(\sg)$. Therefore, we can study the distribution of the Dyck path pattern $DRRR$-matches in all Dyck paths of size $n$ instead of the distribution of the consecutive pattern 132 in all permutations in $\Snn(123)$.

We shall use the first Dyck path recursion to obtain a recursive equation about $\AAAover{123}{132}(t,y,x)$, and we have the following theorem. We will always use the shorthand $A$ for $\AAAover{\la}{\ga}(t,y,x)$ in each computation.

\begin{thm}\label{thm:1}
	The function $\AAxyt{123}{132}$ satisfies the following recursive equation:
	\begin{equation}\label{123132A}
	A=1+t(y(A-1)+1)^2+t^3(x-1)y^2(y(A-1)+1)^3,
	\end{equation}
	and
	\begin{equation}\label{thm1eq}
	\AAAover{123}{132}(t,1,x)|_{t^nx^k}=\frac{1}{k} \sum_{i=k}^{\lfloor\frac{n}{3}\rfloor} (-1)^{i-k} \binom{2n-3i}{n-3i,n+1-i,k-1,i-k}.
	\end{equation}
\end{thm}
\begin{proof} 
	Under the Dyck path bijection $\Psi$, we can take $A$ as a generating function of Dyck paths, where for each path, $t$ tracks the size of the path, $y$ tracks the total number of $RD$ and $RRR$-matches, and $x$ tracks the number of $DRRR$-matches. 
	
	We say that a horizontal segment is \emph{interior} if it is neither the first horizontal segment nor the last horizontal segment. For each Dyck path $P$, we call the path $PD$ (adding an extra step $D$ at the end of the path $P$) the extended path of $P$. 
	For our convenience, we shall define a generating function $A_1{(t,y,x)}$ which is similar to $A$:
	\begin{equation}\label{A1def}
	A_1{(t,y,x)}:=\sum_{n\geq 0}t^n\sum_{P\in\Dn}y^{\{RD,RRR\}\mbox{-mch}(PD)}x^{DRRR\mbox{-mch}(PD)}.
	\end{equation}
	$A_1$ is short for $A_1{(t,y,x)}$, unless stated otherwise. Clearly $A_1$ is a generating function of Dyck paths, where for each path $P$, $t$ tracks the size of $P$, $y$ tracks the total number of $RD$ and $RRR$-matches in the extended path $PD$, and $x$ tracks the number of $DRRR$-matches in the extended path $PD$. In other words, $A_1$ is a generating function of Dyck paths in which we suppose there are some lower horizontal segments after the last step, and the last segment is an interior segment. 
	
	The difference between $A$ and $A_1$ is that, for any nonempty path, $A_1$ counts one more pattern $RD$ than $A$ formed by the last step $R$ and the extended step $D$. Thus we have the following relation between $A$ and $A_1$:
	\begin{equation}\label{AA01}
	A=\frac{A_1-1}{y}+1.
	\end{equation}
	
	Next, we shall formulate a recursion for $A_1$. We use Recursion 1 described in Section \ref{sec:recursions}, i.e.\ we expand the last horizontal segment of a Dyck path $P$. The empty path has a contribution of $1$ to $A_1$. If the last horizontal segment is of length $1$, then the path ends with steps $DR$, and there is a Dyck path structure before the last two steps with a weight of $A_1$, as shown in \fref{Dyck1a}. The weight of the last segment is $ty$ since the last step $R$ and the extended step $D$ form an $RD$ pattern contributing $1$ to the power of $y$. Thus the contribution in this case is $t y A_1$. 
	
	If the last horizontal segment is of length $2$, then the path ends with steps $DRR$, and there are 2 Dyck path structures before the last three steps $DRR$ each with a weight of $A_1$, as shown in \fref{Dyck1b}. The weight of the last segment is $t^2y$ since the last step $R$ and the extended step $D$ form an $RD$ pattern contributing $1$ to the power of $y$. Thus the contribution in this case is $t^2 y A_1^2$. 
	
	If the last horizontal segment is of length $k\geq 3$, then the path ends with steps $DR^k$, and there are $k$ Dyck path structures before the last $k+1$ steps $DR^k$ each with a weight of $A_1$, as shown in \fref{Dyck1c}. The weight of the last segment is $t^k x y^{k-1}$ since the last step $R$ and the extended step $D$ form an $RD$ pattern, and the $R^k$ steps in the last segment contain $k-2$ $RRR$ patterns, contributing $1+k-2$ to the power of $y$. The $DRRR$ steps in the last segment contribute $1$ to the power of $x$. Thus the contribution in this case is $t^k x y^{k-1} A_1^k$. Thus we have
	\begin{eqnarray}\label{A01}
	A_1&=&1+t y A_1 + t^2yA_1^2+t^3xy^2A_1^3 +t^4xy^3A_1^4+\cdots\nonumber\\
	&=&1+tyA_1+t^2yA_1^2+\frac{t^3xy^2A_1^3}{1-tyA_1}.
	\end{eqnarray}
	
	We can get equation (\ref{123132A}) from equation (\ref{A01}) by using the substitution relation in equation  (\ref{AA01}).
	We shall prove equation (\ref{thm1eq}) in Section 5.1.
\end{proof}

Setting $x=0$ and $y=1$ in $\AAxyt{123}{132}$, we have
\begin{equation}\label{motz}
A_{123}^{\underline{132}}(t,1,0)=\frac{1 - t - \sqrt{1-2t-3t^2} }{2t^2},  
\end{equation}
which coincides with the generating function of Motzkin numbers (entry  A001006 in the OEIS \cite{OEIS}). In other words, the number of permutations in $\Snn$ avoiding both the classical pattern $123$ and the consecutive pattern $132$ is equal to the number of Motzkin paths from $(0,0)$ to $(n, 0)$. 

We will generalize our method to compute $A_{123}^{\underline{1m\cdots 2}}(t,y,x)$ for any $m\geq2$ in Section 5.1.

\subsection{The function $\AAxyt{123}{231}$}

For any permutation $\sg\in\Snn(123)$, \lref{p2} implies that 
$\sg_i\sg_{i+1}\sg_{i+2}$ is a $231$-match of $\sg$ if and only if $\sg_i$ and $\sg_{i+2}$ both sit above the first $R$ steps of some horizontal segments of $\Psi(\sg)$, thus
the number of consecutive $231$-matches in $\sg$ is equal to the number of path pattern $DRRD$-matches (looks like \begin{tikzpicture}[scale =.2]\draw[help lines] (0,2) grid (2,0);\draw[ultra thick] (0,2)--(0,1)--(2,1)--(2,0);\end{tikzpicture}) in the Dyck path $\Psi(\sg)$. 
As with the analysis of the generating function $\AAxyt{123}{132}$, we only need to study the distribution of the pattern $DRRD$-matches  in all Dyck paths to get $\AAxyt{123}{231}$.

\begin{thm}\label{thm:2}
	The function $\AAxyt{123}{231}$ satisfies the following recursion:
	\begin{multline}\label{thm2A}
	A=t^2 (y-1)^2 \left(y \left((A-1)^2 x^2 y+x \left((A-1)^3 y^3+(A-1)^2
	y^2+(A-1) y+2 A-1\right)\right.\right.\\
	\left.\left.-((A-1) y+1) (y ((A-2) (A-1) y+2
	A-3)+3)\right)+1\right)+(A-1) y ((y-3) y+3)+1\\
	-t (y-1)^2 ((A-1) y+1) (y (A
	(x+y-2)-x-y+3)-1), 
	\end{multline}
	and
	\begin{equation}\label{thm2eq}
	\AAAover{123}{231}(t,1,x)|_{t^nx^k}=\frac{1}{n} \sum_{i=0}^{\lfloor\frac{3n-1}{5}\rfloor}(-1)^{4n+k+1}\binom{n}{i} \binom{n-2i}{3n-1} \binom{2n-2i}{k}.
	\end{equation}
\end{thm}

\begin{proof}
	Under the Dyck path bijection $\Psi$, we can take $A$ as a generating function of Dyck paths where for each path, $t$ tracks the size of the path, $y$ tracks the total number of $RD$ and $RRR$-matches, and $x$ tracks the number of $DRRD$-matches. 
	
	We define a generating function $A_1$ which is similar to $A$ like Section 3.1:
	\begin{equation}
	A_1{(t,y,x)}:=\sum_{n\geq 0}t^n\sum_{P\in\Dn}y^{\{RD,RRR\}\mbox{-mch}(PD)}x^{DRRD\mbox{-mch}(PD)}.
	\end{equation}
	In $A_1$, for each Dyck path $P$, we let $t$ track the size of $P$,  $y$ track the total number of $RD$ and $RRR$-matches in $PD$, and $x$ track the number of $DRRD$-matches in $PD$.
	
	The difference between $A$ and $A_1$ is that, for any nonempty path $P$, $A_1$ counts one more pattern $RD$ that the last step $R$ and the extended  step $D$ form; if the last segment of $P$ has length 2, then the last segment and the extended step $D$ form a $DRRD$ pattern contributing 1 to the power of $x$. Thus we have the following relation between $A$ and $A_1$,
	\begin{equation}\label{AA02}
	A=1 + \frac{A_1-1}{y} +t^2(1-x)A_1^2.
	\end{equation}
	
	Next, we shall formulate a recursion for $A_1$ using the first Dyck path recursion  by expanding the last horizontal segment of a Dyck path $P$. The empty path has a contribution of $1$. Let $k$ be the size of the last segment, then the case when $k=1$ still has a contribution of $t y A_1$ to $A_1$.
	
	If $k=2$, then the path ends with steps $DRR$, and there are 2 Dyck path structures before the last three steps each with a weight of $A_1$, as shown in \fref{Dyck1b}. The weight of the last segment is $t^2xy$ since the last step $R$ and the extended step $D$ form an $RD$ pattern contributing $1$ to the power of $y$, and the last segment of $P$ and the extended step $D$ form a $DRRD$ pattern contributing 1 to the power of $x$. Thus the contribution in this case is $t^2 xy A_1^2$. 
	
	If $k\geq 3$, then there is no pattern $DRRD$ appearing in the last segment, and the last segment will contribute $0$ to the power of $x$. Similar to the discussion in Section 3.1, the contribution of the case when the size of the last segment is $k\geq 3$ is $t^k y^{k-1} A_1^k$. Thus we have
	\begin{eqnarray}\label{A02}
	A_1&=&1+t y A_1 + t^2xyA_1^2+t^3y^2A_1^3+\cdots\nonumber\\
	&=&1+t y A_1 + t^2xyA_1^2+\frac{t^3y^2A_1^3}{1-tyA_1}.
	\end{eqnarray}
	We can prove equation (\ref{thm2A}) from equation (\ref{A02}) to   by using the substitution relation implied in equation  (\ref{AA02}).
	We shall prove equation (\ref{thm2eq}) in Section 5.2.
\end{proof}
Setting $x=0$ and $y=1$ in $\AAxyt{123}{231}$, the coefficient list $\{a_n\}_{n\geq 0}$ in the Taylor series expansion of $\AAAover{123}{231}(t,1,0)$ is
$\{1, 1, 2, 4, 9, 23, 63, 178, 514, \ldots\}$
which is entry  A135307 in the OEIS \cite{OEIS}, and the formula for $a_n$ is given by $$a_n=\frac{1}{n} \sum_{i=0}^{(n-1)/2}(-1)^i \binom{n}{i}\binom{2n-3i}{n-i+1}.$$

We will generalize our method to compute $A_{123}^{\underline{2m\cdots 31}}(t,y,x)$ for any $m\geq2$ in Section 5.2.

\subsection{The function $\AAxyt{123}{321}$}

By \lref{p2}, the consecutive pattern $321$ of a permutation $\sg\in\Snn(123)$ does not correspond to a single path pattern of $\Psi(\sg)$. We will show the details in the proof of the following theorem.
\begin{thm}\label{thm:3}
	The generating function $\AAxyt{123}{321}$ is given by
	\begin{eqnarray}
	\AAxyt{123}{321}&=&1- \frac{1}{2 t y^2 \left(x^2 (-t) y+x+t\right)^3}
	\left(2 (x-1)^2 t^6 y^3 \left(x^2 y-1\right)^2\right.\nonumber\\
	&&+t^2 \left(-x^2 (2
	x+5) y^2+(6 x+2) y-1\right)+2 t^4 y \left((2-3 x) x^4
	y^3+\left(3 x^2+2 x-2\right) x^2 y^2\right.\nonumber\\
	&&\left.-\left(3
	x^2+x-1\right) y+1\right)+2 (x-1) t^5 y^2 \left(x^5
	y^3-3 x^4 y^2+4 x^2 y-x y-1\right)\nonumber\\\nonumber
	&&+2 t^3 y \left(3 x^4
	y^2-x^3 y-5 x^2 y+x+2\right)+t (4 x y-2)-1\\
	&&\left.+(t (x y-1)-1)^2\sqrt{4 x^2 t^2 y^2-4 x t y-4 t^2 y+1}\right).
	\end{eqnarray}
\end{thm}
\begin{proof}
	We can count consecutive $321$-matches from the horizontal segments of a Dyck path. It is not hard to see from \lref{p2} that the contribution of a horizontal segment to $321$-matches is given by \taref{123c321} if there are more than one horizontal segments in the path.
	
	\begin{table}[ht!]
		\centering
		\vspace*{-3mm}
		\caption{the contribution of a horizontal segment to the power of $x$ in $\AAxyt{123}{321}$.}
		\vspace*{3mm}
		\begin{tabular}{|c|c|c|c|}
			\hline
			size&first segment&interior segment&last segment\\\hline
			$1$&$0$&$1$&$0$\\\hline
			$2$&$0$&$0$&$0$\\\hline
			$3$&$1$&$1$&$0$\\\hline
			$4$&$2$&$2$&$1$\\\hline
			$5$&$3$&$3$&$2$\\\hline
			$\cdots$&$\cdots$&$\cdots$&$\cdots$\\\hline
		\end{tabular}
		\label{table:123c321}
	\end{table}
	
	It is clear from the table that the type and the size of a segment of $\Psi(\sg)$ determine its contribution to consecutive $321$-matches in $\sg\in\Snn(123)$. In this proof, we shall define another $2$ generating functions, $A_0$ and $A_1$, which are similar to $A$. For a permutation $\sg\in\Snn$, we let $\sg0$ denote the length-$(n+1)$ word obtained by adding a $0$ at the end of $\sg$, and let $(n+1)\sg0$ denote the length-$(n+2)$ word obtained by adding an $(n+1)$ before $\sg$ and a $0$ at the end of $\sg$. Then we define
	\begin{eqnarray}
	A_0{(t,y,x)}&:=&\sum_{n\geq 0}t^n\sum_{\sg\in\Snn(123)}y^{\des((n+1)\sg0)}x^{\textrm{321-mch}((n+1)\sg0)},\\
	A_1{(t,y,x)}&:=&\sum_{n\geq 0}t^n\sum_{\sg\in\Snn(123)}y^{\des(\sg0)}x^{\textrm{321-mch}(\sg0)}.
	\end{eqnarray}
	Both $A_0$ and $A_1$ track the same statistics as $A$. $A_0$ is the generating function of Dyck paths that we take every segment as an interior segment, and $A_1$ is the generating function of Dyck paths that we suppose the last segment is an interior segment. 
	
	\begin{figure}[ht]
		\centering
		\vspace{-1mm}
		\begin{tikzpicture}[scale =.5]
		\draw[dashed, thick] (-.5,3.5) -- (.5,3.5) -- (1,3);
		\draw[dashed, thick] (3,0) -- (3,1);
		\draw (1,1) rectangle (3,3);
		\node at (2,2) {$A_0$};
		\path (6,0);
		\end{tikzpicture}	
		\begin{tikzpicture}[scale =.5]
		\draw[dashed, thick] (3,0) -- (3,1);
		\draw (1,1) rectangle (3,3);
		\node at (2,2) {$A_1$};
		\path (6,0);
		\end{tikzpicture}
		\begin{tikzpicture}[scale =.5]
		\draw (1,1) rectangle (3,3);
		\node at (2,2) {$A$};
		\path (5,0);
		\end{tikzpicture}
		\vspace{-3mm}
		\caption{Dyck paths in the generating functions $A_0$, $A_1$ and $A$.}
		\label{fig:A0A1A}
	\end{figure}
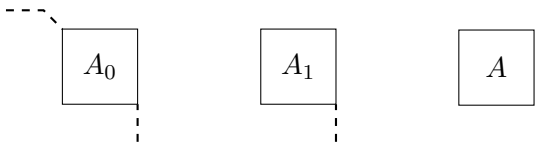
	In other words, in each path $P$ in the generating function $A_0$, we suppose that there are some $R$ steps before $P$ and a $D$ step following $P$ that do not contribute to the size of the path but contribute to the pattern matches; and in each path $P$ in the generating function $A_1$, there is a $D$ step after $P$ that does not contribute to the size but contributes to the pattern matches. \fref{A0A1A} gives sketches of paths in $A_0$, $A_1$ and $A$.

	First we shall compute a recursion for $A_0$ since it only consists of interior segments. For each segment in $A_0$, we have its contribution to descents and  $321$-matches by \lref{p2}, summarized in \taref{123A0}. 
	
	\begin{table}[ht!]
		\centering
		\vspace*{-3mm}
		\caption{the contribution of horizontal segments in $A_0$.}
		\vspace*{3mm}
		\begin{tabular}{|c|c|c|}
			\hline
			size&contribution to descents&contribution to $321$-matches\\
			& (power of $y$)& (power of $x$)\\\hline
			$1$&$1$&$1$\\\hline
			$2$&$1$&$0$\\\hline
			$3$&$2$&$1$\\\hline
			$4$&$3$&$2$\\\hline
			$\cdots$&$\cdots$&$\cdots$\\\hline
		\end{tabular}
		\label{table:123A0}
	\end{table}
	
	Then, by applying the first Dyck path recursion to expand the last horizontal segment of a Dyck path, we have a recursion for $A_0$ according to \taref{123A0}:
	\begin{eqnarray}\label{123A0}
	A_0&=&1+txyA_0+t^2yA_0^2+t^3xy^2A_0^3+t^4x^2y^3A_0^4+\cdots\nonumber\\
	&=&1+txyA_0+\frac{t^2yA_0^2}{1-txyA_0}.
	\end{eqnarray}
	
	Next, we want to compute a recursion for $A_1$ in terms of $A_0$ and $A_1$. We use the third Dyck path recursion which expands the last horizontal segment before the first return that is of size $k$. When $k=1$, there are only 2 steps, $DR$,  before the first return. There is a Dyck path structure $P_1$ after the first return with a weight of $A_0$ since all of the segments in $P_1$ are interior segments. The contribution to $A_1$ of this case is $tyA_0$ as shown in \fref{A11}. 
	
	When $k=2$, the path ends with steps $DRR$ before the first return. There is a Dyck path structure $P_1$ after the first return with a weight of $A_0$ for the same reason. By the third Dyck path recursion, there is another Dyck path structure $P_2$ above the last horizontal segment before the first return. The weight of $P_2$ is $A_1$ since it contains the first segment of the whole path. The contribution of this case is $t^2yA_0A_1$ as shown in \fref{A12}.
	
	When $k=3$, the path ends with steps $DRRR$ before the first return. There is a Dyck path structure $P_1$  after the first return with a weight of $A_0$. By the third Dyck path recursion, there are two Dyck path structures $P_2,\ P_3$ (counting from top to bottom) above the last horizontal segment before the first return. If $P_2$ is not empty, then it has a weight of $A_1-1$ since it contains the first segment, and the weight of $P_3$ is $A_0$. If $P_2$ is empty, then it has a weight of $1$, and $P_3$ has a weight of $A_1$ since it contains the first segment of the whole path. The contribution of this case is $t^3xy^2A_0(A_0(A_1-1)+A_1)$ as shown in \fref{A13}.

	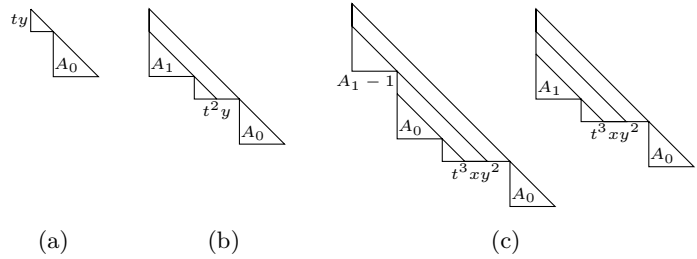
\begin{figure}[ht]
		\centering
		\vspace{-1mm}
		\subfigure[\label{fig:A11}]{\begin{tikzpicture}[scale =.3]
			\draw (0,3)--(0,2)--(1,2)--(1,0)--(3,0)--(0,3);
			\node at (1.6,.5) {\tiny $A_0$};
			\node at (-.5,2.5) {\tiny $ty$};
			\path (4,-6);
			\path (-2,-6);
			\end{tikzpicture}
		}
		\subfigure[\label{fig:A12}]{\begin{tikzpicture}[scale =.3]
			\draw (0,6)--(0,3)--(2,3)--(2,2)--(4,2)--(4,0)--(6,0)--(0,6)--(0,5)--(3,2);
			\node at (.6,3.5) {\tiny $A_1$};
			\node at (4.6,.5) {\tiny $A_0$};
			\node at (3,1.5) {\tiny $t^2y$};
			\path (7,-3);
			\end{tikzpicture}
		}
		\subfigure[\label{fig:A13}]{\begin{tikzpicture}[scale =.3]
			\draw (0,9)--(0,6)--(2,6)--(2,3)--(4,3)--(4,2)--(7,2)--(7,0)--(9,0)--(0,9)--(0,8)--(6,2)--(5,2)--(2,5);
			\node at (.6,5.5) {\tiny $A_1-1$};
			\node at (2.6,3.5) {\tiny $A_0$};
			\node at (7.6,.5) {\tiny $A_0$};
			\node at (5.5,1.5) {\tiny $t^3xy^2$};
			\end{tikzpicture}\hspace*{-5mm}
			\begin{tikzpicture}[scale =.3]
			\draw (2,7)--(2,3)--(4,3)--(4,2)--(7,2)--(7,0)--(9,0)--(2,7)--(2,6)--(6,2)--(5,2)--(2,5);
			\node at (2.6,3.5) {\tiny $A_1$};
			\node at (7.6,.5) {\tiny $A_0$};
			\node at (5.5,1.5) {\tiny $t^3xy^2$};
			\path (2,-2);
			\end{tikzpicture}}
		\vspace{-3mm}
		\caption{contributions to $A_1$ when the last segment before the first return is of size $1,2,3$.}
		\label{fig:A1}
	\end{figure}
	
	Similar to the case when $k=3$, if the last horizontal segment before the first return is of size $k\geq3$, then the total contribution is $t^kx^{k-2}y^{k-1}A_0(A_0^{k-2}(A_1-1)+\cdots+A_0(A_1-1)+A_1)$. It follows that
	\begin{eqnarray}\label{123A1}
	A_1&=&1+tyA_0+t^2yA_0A_1+t^3xy^2A_0(A_0(A_1-1)+A_1)+\cdots\nonumber\\
	&=&1+tyA_0+\frac{A_0-A_1}{A_0-1}\frac{t^2yA_0}{1-txy}+\frac{A_1-1}{A_0-1}\frac{t^2yA_0^2}{1-txyA_0}.
	\end{eqnarray}
	
	Finally for $A$, we use the first Dyck path recursion by expanding the last horizontal segment. The last segment of size $k$ contains $k-2$ pattern $RRR$-matches and $k-3$ pattern $RRRR$-matches. Thus it contributes nothing to the power of $y$ and $x$ when $k<3$, and it has a weight of $t^kx^{k-3}y^{k-2}$ when $k\geq 3$.
	
	Above the last segment, each Dyck path structure has a weight of either $A_0$ or $A_1$ based on whether it contains the first segment of the whole path or not. Similar to our analysis in the case when $k\geq 3$ in $A_1$, the contribution of the part of path above the last segment is
	\begin{equation*}
	A_0^{k-1}(A_1-1)+A_0^{k-2}(A_1-1)+\cdots+A_0^{2}(A_1-1)+A_0(A_1-1)+A_1.
	\end{equation*}
	Combining the weight of the last segment and the part above the last segment, we can rewrite $A$ in terms of $A_1$ and $A_0$ as follows:
	\begin{eqnarray}\label{123A}
	A&=&1+tA_1+t^2(A_0(A_1-1)+A_1)+t^3y(A_0^2(A_1-1)+A_0(A_1-1)+A_1)+\cdots\nonumber\\
	&=&1+tA_1+t^2(A_0 A_1 -A_0 +A_1) + \frac{A_0-A_1}{A_0-1}\frac{t^3y}{1-txy}+\frac{A_1-1}{A_0-1}\frac{t^3yA_0^3}{1-txyA_0}.
	\end{eqnarray}
	Solving equations (\ref{123A0}),(\ref{123A1}),(\ref{123A}) for $A_0$, $A_1$ and $A$ gives the formula for $A$ in this theorem. Notice that $A$ is a root of some quadratic equation.
\end{proof}

Setting $x=0$ and $y=1$ in $\AAxyt{123}{321}$, the coefficient list in the Taylor series expansion of $\AAAover{123}{321}(t,1,0)$ is
$\{1, 1, 2, 4, 9, 23, 63, 178, 514, \ldots\}$
which is the same as entry A059347 in the OEIS \cite{OEIS} after the third term.

\subsection{The function $\AAxyt{132}{123}$}
In this subsection, we shall study the distribution of consecutive $123$-matches in the permutations in $\Snn(132)$. Similar to the previous subsections, we can apply the bijection $\Phi$ of Krattenthaler \cite{Kr} to translate the consecutive patterns in permutations  $\sg\in\Snn(132)$ into path patterns in Dyck paths $\Phi(\sg)$. The cases in $\Snn(132)$ tend to be more straightforward, and we can find the generating function $A$ directly without defining the auxiliary functions  $A_0$ or $A_1$.

Given any permutation $\sg\in\Snn(132)$, \lref{p1} implies the descent positions of $\sg$ only appear between two consecutive horizontal segments in the corresponding Dyck path $\Phi(\sg)$,
thus the number of descents of $\sg$ is one less than the number of horizontal segments in $\Phi(\sg)$, i.e.\ $\des(\sg)$ is equal to the number of occurrences of the path pattern $RD$ in $\Phi(\sg)$. For the same reason, the consecutive pattern $123$ corresponds to the path pattern $RRR$. Now we are ready to prove the following theorem.

\begin{thm}\label{thm:4}
	The generating function $\AAxyt{132}{123}$ is given by
	\begin{multline}\label{132123A}
	\AAxyt{132}{123}=\\\frac{\sqrt{x^2 t^2+2 x t (t y-1)+t^2 (y-4) y-2 t y+1}+x t
		\left(2 t y^2-2 (t+1) y+1\right)-2 t^2 (y-1) y+t y-1}{2 t y (x (t y-1)-t y)},
	\end{multline}
	and
	\begin{equation}\label{thm4eq}
	\AAAover{132}{123}(t,1,x)|_{t^nx^k}= \frac{1}{n+1} \sum_{i=0}^{\lfloor\frac{n}{2}\rfloor} \sum_{j=0}^{n+1-i} (-1)^{2j} \binom{n+1}{i} \binom{i+k-1}{k} 
	\binom{2n-3i-2j-k}{n-i} .
	\end{equation}
\end{thm}
\begin{proof}
	According to the previous discussion, $A$ is a generating function of Dyck paths, where for each path $P$, $t$ tracks the size of $P$, $y$ tracks the number of $RD$-matches in $P$, and $x$ tracks the number of $RRR$-matches in $P$. Thus, a horizontal segment of size $1$ or $2$ contributes nothing to the power of $x$, and a horizontal segment of size $k\geq3$ contributes $k-2$ to the power of $x$. This allows us to derive a recursive equation about $A$ by expanding the last horizontal segment.
	
	Now we assume that $P$ is a Dyck path where the last horizontal segment is of size $k$. When $k$ is $0$, $P$ is the empty path which has a contribution of $1$ to the function $A$. When $k=1$, $P$ ends with steps $DR$ and there is a Dyck path structure $P_1$ above the last horizontal segment with a weight of $(y(A-1)+1)$ since there is an extra descent between $P_1$ and the last segment if $P_1$ is not empty. When $k\geq2$, the path ends with steps $DR^k$. The last segment has a weight of $t^k x^{k-2}$ since it contains $k-2$ pattern $RRR$-matches. There are $k$ Dyck path structures $P_1,\ldots, P_k$ above the last segment, and each $P_i$ has a weight of $(y(A-1)+1)$ since $P_i$ will create an extra descent with the horizontal segment following $P_i$ if $P_i$ is not empty. Thus the contribution of the case when $k\geq2$ is $t^kx^{k-2}(y(A-1)+1)^k$. It follows that
	\begin{eqnarray}\label{13212333}
	A&=&1+ t(y(A-1)+1)+ t^2(y(A-1)+1)^2 +t^3x(y(A-1)+1)^3+\cdots\nonumber\\
	&=&1+ t(y(A-1)+1)+ \frac{t^2(y(A-1)+1)^2}{1-tx(y(A-1)+1)}.
	\end{eqnarray}
	Equation (\ref{13212333}) can be adjusted into a quadratic equation about $A$, and one can solve the quadratic equation to prove equation (\ref{132123A}).
	
	We shall prove equation (\ref{thm4eq}) in Section 6.1.
\end{proof}

Setting $x=0$ and $y=1$ in $\AAxyt{132}{123}$, we have
$A_{132}^{\underline{123}}(t,1,0)=A_{123}^{\underline{132}}(t,1,0)$, which is again the generating function of Motzkin numbers (entry A001006 in the OEIS \cite{OEIS}). 

We will generalize our method to compute $A_{132}^{\underline{1\cdots m}}(t,y,x)$ in Section 6.1.

\subsection{The function $\AAxyt{132}{213}$}
\begin{thm}\label{thm:6}
	The generating function $\AAxyt{132}{213}$ is given by
	\begin{multline}
	\AAxyt{132}{213}=\\\frac{-\sqrt{t \left(2 y \left(x (t-1) t-t^2-1\right)+t y^2 (-x
			t+t+1)^2+t-2\right)+1}+(x-1) t^2 y (2 y-1)+t (y-1)+1}{2 t y ((x-1) t
		y+1)}.
	\end{multline}
\end{thm}
\begin{proof}
	Similar to the discussion in Section 3.4, we shall still use the first Dyck path recursion to compute $\AAxyt{132}{213}$. Given any Dyck path $P$ where the last horizontal segment is of size $k$, we shall expand the last segment as shown is \fref{132c213}. 
	By \lref{p1}, there is a length $k$ increasing sequence lying above the last horizontal segment, and there are $k$ Dyck path structures $P_1,\ldots,P_k$ above the last segment. Other than inside each of the $k$ Dyck path structures, the only possible location of consecutive $213$-match is at the junction of the last Dyck path structure $P_k$ (if $P_k$ is nonempty) and the last segment (if $k$ is at least 2) which is in the circled area of \fref{132c213}. 
	\begin{figure}[ht]
		\centering
		\begin{tikzpicture}[scale =.4]
		\draw[very thick] (5.5,2) -- (5.5,5.5) --(2,5.5);
		\draw[very thick] (6.5,2) -- (6.5,8.5) --(0,8.5);
		\fill[black!30!white] (4,2) rectangle (5,3);
		\fill[black!30!white] (5,5) rectangle (6,6);
		\fill[black!30!white] (6,8) rectangle (7,9);
		\draw (0,9)--(0,6)--(2,6)--(2,3)--(4,3)--(4,2)--(7,2)--(-2,11)--(-2,9)--(0,9)--(0,8)--(6,2)--(5,2)--(2,5);
		\draw[thick, red] (4,4.2) circle (2.2);
		\node at (.5,6.5) {\footnotesize $\cdots$};
		\node at (2.5,3.5) {\footnotesize $P_k$};
		\node at (-1.5,9.5) {\footnotesize $P_1$};
		\end{tikzpicture}
		\vspace{-5mm}
		\caption{using Recursion 1 to compute $\AAxyt{132}{213}$.}
		\label{fig:132c213}
	\end{figure}
	
	The total contribution of the cases when $k$ is $0$ or $1$ is still $1+t(y(A-1)+1)$. When $k$ is at least 2, the weight of each $P_i$ for $i=1,\ldots,k-1$ is $(y(A-1)+1)$. If $P_k$ is not empty, then the total weight of $P_k$ and the last segment is $t^kxy(A-1)$ since $P_k$ and the last segment create an extra consecutive $213$-match; if $P_k$ is empty, then the weight of $P_k$ and the last segment is $t^k$. Thus, the total weight of $P_k$ and the last segment is $t^k(xy(A-1)+1)$, and the contribution of the case when the last segment is of size $k\geq2$ is $t^k(xy(A-1)+1)(y(A-1)+1)^{k-1}$.
	
	Summing over all the cases, we have the following recursion for $A$:
	\begin{eqnarray}\label{132213}
	A&=&1+ t(y(A-1)+1)+t^2(y(A-1)+1)(x y(A-1)+1) \nonumber\\
	&&+t^3(y(A-1)+1)^2(x y(A-1)+1) +\cdots\nonumber\\
	&=&1+ t(y(A-1)+1)+\frac{t^2(y(A-1)+1)(x y(A-1)+1)}{1-t(y(A-1)+1)}.
	\end{eqnarray}
	Equation (\ref{132213}) can be adjusted into a quadratic equation about $A$, and one can solve the quadratic equation to prove \tref{6}.
\end{proof}
Setting $x=0$ and $y=1$ in $\AAxyt{132}{213}$, the coefficient list in the Taylor series expansion of $\AAAover{132}{213}(t,1,0)$ is
$\{1,1, 2, 4, 9, 22, 57, 154, 429, 1223, \ldots\}$
which is entry  A105633 in the OEIS \cite{OEIS}.

\subsection{The function $\AAxyt{132}{231}$}
\begin{thm}\label{thm:5}
	The generating function
	\begin{equation}\label{5A}
	\AAxyt{132}{231}=\frac{2 t x y-t y-t+1-\sqrt{-4 t^2 x y+t^2 y^2+2 t^2 y+t^2-2 t y-2 t+1}}{2 t x y},
	\end{equation}
	and
	\begin{equation}\label{thm5eq}
	\AAAover{132}{231}(t,1,x)|_{t^nx^k}=\frac{1}{n} \binom{n}{k} \sum_{i=0}^{n-k} (-1)^{4k+i+n+1}\binom{2i-n}{n+1-4k}.
	\end{equation}
\end{thm}
\begin{proof}
	The best way we know to prove this theorem is to use the structure of 132-avoiding permutations (implied by the second Dyck path recursion). Given any permutation $\sg\in\Snn(132)$ such that the number $n$ is in the \thn{k} position, we break $\sg$ before and after the position of $n$. The numbers to the left of $n$ form a 132-avoiding permutation $P_1$ of the set $\{n-k+1,\ldots,n-1\}$, and the numbers to the right of $n$ form a 132-avoiding permutation $P_2$ of the set $\{1,\ldots,n-k\}$, which is shown in \fref{132perm}.
	\begin{figure}[ht]
		\centering
		\vspace{-1mm}
		\begin{tikzpicture}[scale =.5]
		\draw (0,0) rectangle (6.4,6.2);
		\draw[very thick] (0,2) rectangle (4,6);
		\draw[very thick] (4.4,0) rectangle (6.4,2);
		\draw[fill] (4.2,6.2) circle [radius=0.2];
		\filll{2}{4}{P_1};\filll{5.4}{1}{P_2};
		\draw (4.2,6.7) node {\footnotesize $n$};
		\end{tikzpicture}
		\caption{breaking a permutation $\sg\in\Snn(132)$ at the position of $n$.}
		\label{fig:132perm}
	\end{figure}
	
	The case when $P_2$ is empty contributes $tA$ to the generating function $A$, and the case when $P_1$ is empty while $P_2$ is not empty contributes $ty(A-1)$ to $A$ since there is a descent between $n$ and $P_2$. If neither $P_1$ nor $P_2$ is empty, then a consecutive pattern $231$ appears at the position of $n$, and the contribution of this case is $txy(A-1)^2$. Thus we have:
	\begin{equation}\label{132231}
	A=1+t(A+y(A-1))+txy(A-1)^2.
	\end{equation}
	Equation (\ref{5A}) can be obtained by solving the quadratic equation (\ref{132231}).
	
	We shall prove equation (\ref{thm5eq}) in Section 6.4.
\end{proof}

{\bf Remark. }We can also use the first Dyck path recursion to compute $\AAxyt{132}{231}$ which will help in Section 4.2 when we compute a generating function tracking multiple patterns.  According to \lref{p1}, the number of consecutive $231$-matches in $\sg\in\Snn(132)$ is equal to the number of horizontal segments of $\Phi(\sg)$ of size at least 2 which are not the last segment.
Let $A_1$ be the generating function that we suppose the last segment is an interior segment similar to (\ref{A1def}), then by expanding the last segment of a path, we have the following recursion for $A_1$:
\begin{equation}\label{132231a}
A_1=1+tyA_1+t^2x yA_1^2 + t^3xyA_1^3+\cdots=1+tyA_1+\frac{t^2xyA_1^2}{1-tA_1}.
\end{equation}
For the function $A$, we use the first Dyck path recursion  again to obtain
\begin{equation}\label{132231b}
A=1+tyA_1+t^2yA_1^2 + \cdots=1+\frac{tyA_1}{1-tA_1}.
\end{equation}
One can solve equations (\ref{132231a}) and (\ref{132231b}) to prove \tref{5}.

Setting $x=0$ and $y=1$ in $\AAxyt{132}{231}$, the coefficient list $\{a_n\}_{n\geq 0}$ in the Taylor series expansion is given by $a_0=1,\ a_n=2^{n-1}$ if $n>0$, which is the number of permutations with no peaks. 

We will generalize our method to compute $A_{132}^{\underline{2\cdots m1}}(t,1,x)$ in Section 6.4.

\section{Tracking multiple patterns --- $\AAA{\la}{\underline{\{\ga_1,\ldots,\ga_s\}}}{t,y,x_1,\ldots,x_s}$}

As we mentioned in Section 2, $\la\text{-mch}(\sg)=0$ for any $\sg\in\Snn(\la)$. Since $|\Sn{3}|=6$, permutations in $\Snn(\la)$ for $\la\in\Sn{3}$ only contain $5$ kinds of consecutive patterns of length $3$ except $\la$. Suppose that $\Sn{3}=\{\la,\ga_1,\ga_2,\ga_3,\ga_4,\ga_5\}$ and $\sg\in\Snn(\la)$. 
If $\ga_i\text{-mch}(\sg)=\alpha_i$ for $i=1,\ldots,5$, then  $\alpha_5=n-2-\alpha_1-\alpha_2-\alpha_3-\alpha_4$ since any permutation of length $n$ has $n-2$ consecutive patterns of length $3$. 

Thus for any $\la\in\Sn{3}$, if we have obtained a generating function for $\Snn(\la)$ while tracking $4$ consecutive patterns of length $3$, then we can modify the function to get a new generating function tracking all the $5$ patterns of length $3$. By symmetry, we only need to consider the cases when $\la=123$ or $\la=132$.

When $\la=132$, we are able to compute the generating function tracking all the $5$ patterns of length $3$ and the descent statistic. When $\la=123$, we have the recursion for the generating function tracking $3$ patterns of length $3$ and the descent statistic.

\subsection{The function $\AAA{123}{\underline{\{132,231,321\}}}{t,y,x_1,x_2,x_3}$}

\begin{thm}\label{thm:7}
	The function $\AAA{123}{\underline{\{132,231,321\}}}{t,y,x_1,x_2,x_3}$ satisfies the recursion
	\begin{multline}\label{123}
	A=1/(x_3 (x_1-x_2) (A_0 x_3^2
	t^2 y^2 (A_0 t^2 y
	(x_1-x_2)-1)-(A_0+1) x_3 t y
	(A_0 t^2 y (x_1-x_2)-1)
	+A_0
	x_1 t^2 y-1))\\
	\cdot(x_1 (x_3^2 t y (A_0^2 t^2
	(x_2 t^2 y^2+y (3 x_2 t+2
	x_2+t+1)+1)+A_0 ((x_2+1) t^3 y+t^2 (2
	x_2 y+y+2)+t+1)\\
	+t+1)+x_2 t^2 (A_0^2
	(-t) y+A_0-1)+A_0^2 x_3^4 t^5 y^3-A_0
	x_3^3 t^2 y^2 (A_0 (x_2+1) t^3 y+t^2 (A_0
	(2 x_2 y+y+2)+1)\\
	+t+1)
	+x_3 (A_0 x_2
	t^4 y (A_0-y)-A_0 t^3 y (A_0
	x_2+x_2+1)-A_0 t^2 (x_2
	y+y+1)-t-1))
	+x_2 (-A_0^2 x_3^4 t^5
	y^3\\
	-x_3^2 t y (A_0^2 (2 t-1) t^2 y+A_0 (t^3
	y+t^2 (y+2)+t+1)+t+1)
	+x_3 (t^2 (A_0^2
	(-y)+y+1)+A_0 (A_0+1) t^3 y+t+1)\\
	+A_0
	x_3^3 t^2 y^2 (A_0 t^3 y+(A_0+1)
	t^2+t+1)
	+(A_0-1) t)+x_3 t (x_3 t y-1)
	(A_0^2 x_3 t y (x_3 t
	y-1)+A_0-1)\\
	+A_0 x_1^2 t^2 y (-A_0
	x_2 t^2+A_0 x_3^3 t^2 y^2-(A_0+1) x_3^2 t
	y+x_3)
	+A_0 x_2^2 x_3 t^3 y
	(A_0 x_3^2 t (t+1) y^2\\
	-x_3 y (A_0 t^2
	y+2 A_0 t+A_0+t+1)+(A_0+1) t
	y+1)),
	\end{multline}
	where
	\begin{equation}\label{1231A0}
	A_0=1+ t x_3 y A_0 + t^2 x_1 y A_0^2 + \frac{ t^3 x_2 x_3 y^2 A_0^3}{1 -  t x_3 y A_0}.
	\end{equation}
\end{thm}
The right hand side of (\ref{123}) is a rational function of $A_0$, where both the numerator and the denominator are degree $2$ polynomials in $A_0$.
\begin{proof}
	The proof of this theorem is similar to \tref{3}. It is not hard to see from \lref{p2} that the contribution of a horizontal segment to $132,231,321$-matches and descents is given by \taref{123table} if there are at least two horizontal segments in the path.
	
	\begin{table}[ht!]
		\centering
		\vspace*{-3mm}
		\caption{the contribution of a horizontal segment to $132,231,321$-matches and descents.}
		\vspace*{3mm}
		\begin{tabular}{|c|c|c|c|}
			\hline
			size&first segment&interior segment&last segment\\\hline
			$1$&$0,0,0,0$&$0,0,1,1$&$0,0,0,0$\\\hline
			$2$&$0,1,0,0$&$0,1,0,1$&$0,0,0,0$\\\hline
			$3$&$1,0,1,1$&$1,0,1,2$&$1,0,0,1$\\\hline
			$4$&$1,0,2,2$&$1,0,2,3$&$1,0,1,2$\\\hline
			$5$&$1,0,3,3$&$1,0,3,4$&$1,0,2,3$\\\hline
			$\cdots$&$\cdots$&$\cdots$&$\cdots$\\\hline
		\end{tabular}
		\label{table:123table}
	\end{table}
	
	We shall define $2$ generating functions, $A_0$ and $A_1$, similar to the functions $A_0$ and $A_1$ defined in Section 3.3. Both $A_0$ and $A_1$ track the same statistics as $A$. $A_0$ is a generating function of Dyck paths where we take every segment as an interior segment, and $A_1$ is a generating function of Dyck paths where we suppose  the last segment is an interior segment. 
	
	We shall begin with deriving a recursive expression for $A_0$ since there are only interior segments. By expanding the last horizontal segment of $A_0$, we have the following recursion which proves equation (\ref{1231A0}).
	\begin{eqnarray}\label{1231A00}
	A_0&=&1+ t x_3 y A_0 + t^2 x_1 y A_0^2 + t^3 x_2 x_3 y^2 A_0^3 + t^4 x_2 x_3^2 y^3 A_0^4+\cdots\nonumber\\
	&=&1+ t x_3 y A_0 + t^2 x_1 y A_0^2 + \frac{ t^3 x_2 x_3 y^2 A_0^3}{1 -  t x_3 y A_0}.
	\end{eqnarray}
	
	Next, we want to compute a recursion for $A_1$. We use the third Dyck path recursion which expands the last horizontal segment before the first return which is of size $k$. If $k=1$, then the only steps before the first return are $DR$. There is a Dyck path structure $P_1$ after the first return with a weight of $A_0$, and the contribution of this case is $tyA_0$ as shown in \fref{A21}. 
	
	When $k=2$, the path ends with steps $DRR$ before the first return. There is a Dyck path structure $P_1$  after the first return with a weight of $A_0$. By the third Dyck path recursion, there is another Dyck path structure $P_2$ above the last horizontal segment before the first return. The weight of $P_2$ is $A_1$ since it contains the first segment of the whole path. The weight of the last segment before the first return is $t^2x_2y$, and the contribution of this case is $t^2x_2yA_0A_1$ as shown in \fref{A22}.
	
	When $k=3$, the path ends with steps $DRRR$ before the first return. There is a Dyck path structure $P_1$  after the first return with a weight of $A_0$. There are two Dyck path structures $P_2,\ P_3$ (counting from top to bottom) above the last horizontal segment before the first return. If $P_2$ is not empty, then it has a weight of $A_1-1$ since it contains the first segment, and the weight of $P_3$ is $A_0$. If $P_2$ is empty, then it has a weight of $1$, and $P_3$ has a weight of $A_1$ since it contains the first segment of the whole path. The weight of the last segment before the first return is $t^3x_1x_3y^2$, and the contribution of this case is $t^3x_1x_3y^2A_0(A_0(A_1-1)+A_1)$ as shown in \fref{A23}.
	
	\begin{figure}[ht]
		\centering
		\vspace{-1mm}
		\subfigure[\label{fig:A21}]{\begin{tikzpicture}[scale =.3]
			\draw (0,3)--(0,2)--(1,2)--(1,0)--(3,0)--(0,3);
			\node at (1.6,.5) {\tiny $A_0$};
			\node at (-.5,2.5) {\tiny $ty$};
			\end{tikzpicture}}
		\subfigure[\label{fig:A22}]{\begin{tikzpicture}[scale =.3]
			\draw (0,6)--(0,3)--(2,3)--(2,2)--(4,2)--(4,0)--(6,0)--(0,6)--(0,5)--(3,2);
			\node at (.6,3.5) {\tiny $A_1$};
			\node at (4.6,.5) {\tiny $A_0$};
			\node at (3,1.5) {\tiny $t^2x_2y$};
			\end{tikzpicture}}
		\subfigure[\label{fig:A23}]{\begin{tikzpicture}[scale =.3]
			\draw (0,9)--(0,6)--(2,6)--(2,3)--(4,3)--(4,2)--(7,2)--(7,0)--(9,0)--(0,9)--(0,8)--(6,2)--(5,2)--(2,5);
			\node at (.6,5.5) {\tiny $A_1-1$};
			\node at (2.6,3.5) {\tiny $A_0$};
			\node at (7.6,.5) {\tiny $A_0$};
			\node at (5.3,1.5) {\tiny $t^3x_1x_3y^2$};
			\end{tikzpicture}\hspace*{-5mm}
			\begin{tikzpicture}[scale =.3]
			\draw (2,7)--(2,3)--(4,3)--(4,2)--(7,2)--(7,0)--(9,0)--(2,7)--(2,6)--(6,2)--(5,2)--(2,5);
			\node at (2.6,3.5) {\tiny $A_1$};
			\node at (7.6,.5) {\tiny $A_0$};
			\node at (5.3,1.5) {\tiny $t^3x_1x_3y^2$};
			\end{tikzpicture}}
		\caption{the function $\AAA{123}{\underline{\{132,231,321\}}}{t,y,x_1,x_2,x_3}$: contributions to $A_1$ when $k=1,2,3$.}
		\label{fig:A2}
	\end{figure}
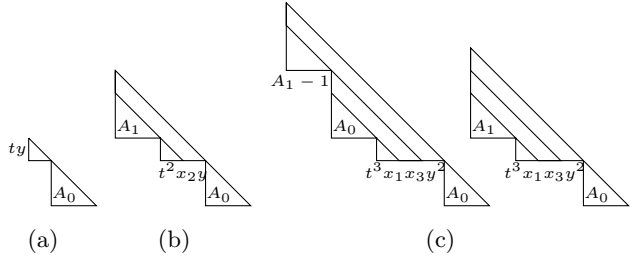
	
	In general, if the last horizontal segment before the first return is of size $k\geq3$, then the contribution is
	\begin{equation*}
	t^kx_1x_3^{k-2}y^{k-1}A_0(A_0^{k-2}(A_1-1)+\cdots+A_0(A_1-1)+A_1).
	\end{equation*}
	It follows that
	\begin{eqnarray}\label{1231A1}
	A_1&=&1+tyA_0+ t^2x_2yA_0A_1 +t^3 x_1x_3y^2A_0(A_0(A_1-1)+A_1)+\cdots\nonumber\\
	&=&1+tyA_0+ t^2x_2yA_0A_1 - \frac{t^3x_1x_3y^3A_0(A_0-A_1-A_0A_1+tx_3yA_0A_1)}{(1-tx_3y)(1-tx_3yA_0)}.
	\end{eqnarray}
	
	Finally for $A$, we use the first Dyck path recursion by expanding the last horizontal segment. Suppose that the last segment is of size $k$. Similar to Section 3.3, the total contribution when $k<3$ is still $1 + t A_1+t^2(A_0(A_1-1)+A_1)$. When $k\geq 3$, the weight of the last segment is $t^k x_1 x_3^{k-3} y^{k-2}$, and the weight of the part  above the last segment is
	\begin{equation*}
	A_0^{k-1}(A_1-1)+A_0^{k-2}(A_1-1)+\cdots+A_0^{2}(A_1-1)+A_0(A_1-1)+A_1.
	\end{equation*}
	Thus, the contribution of the case when the last segment is of size $k\geq 3$ is
	\begin{equation*}
	t^kx_1 x_3^{k-3}y^{k-2}(A_0^{k-1}(A_1-1)+\cdots+A_0(A_1-1)+A_1).
	\end{equation*}
	It follows that
	\begin{eqnarray}\label{1231A}
	A&=&1 + t A_1+t^2(A_0(A_1-1)+A_1) + \sum_{k\geq 3}t^kx_1 x_3^{k-3}y^{k-2}(A_0^{k-1}(A_1-1)+\cdots+A_0(A_1-1)+A_1)\nonumber\\
	&=&1 + t A_1+t^2(A_0(A_1-1)+A_1) +\frac{t^3x_1y(-A_0+A_1+(tx_3y+1)A_0(A_0-A_1+A_0A_1))}{(1-tx_3y)(1-tx_3yA_0)}	.
	\end{eqnarray}
	Using equations (\ref{1231A00}),(\ref{1231A1}) and (\ref{1231A}), one can express $A$ in terms of $A_0$ to prove this theorem.
\end{proof}

We can compute the Taylor series expansion of $\AAA{123}{\underline{\{132,231,321\}}}{t,y,x_1,x_2,x_3}$ using Mathematica as follows.
\begin{multline}
\AAA{123}{\underline{\{132,231,321\}}}{t,y,x_1,x_2,x_3}=1+t+{t}^{2} \left( 1+y \right) +{t}^{3} \left( {x_3}{y}^{2}+{x_1}y+{x_2}y+2y+1 \right) 
+{t}^{4} \left( {{x_3}}^{2}{y}^{3}\right.\\
\left.+2{x_1}{x_3}{y}^{2}+{x_2}{x_3}{y}^{2}+3{x_1}{y}^{2}+2{x_2}{y}^{2}+3{x_3}{y}^{2}+3{x_1}y+
5{x_2}y+3y+1 \right) 
+{t}^{5} \left( {{x_3}}^{3}{y}^{4}+3{x_1}{{x_3}}^{2}{y}^{3}\right.\\
\left.+{x_2}{{x_3}}^{2}{y}^{3}+13{x_1}{x_3}{y}^{3}+5{x_2}{x_3}{y}^{3}+4{{x_3}}^{2}{y}^{3}
+3{{x_1}}^{2}{y}^{2}+10{x_1}{x_2}{y}^{2}
+10{x_1}{x_3}{y}^{2}+3{{x_2}}^{2}{y}^{2}\right.\\
\left.+7{x_2}{x_3}{y}^{2}+12{x_1}{y}^{2}+15{x_2}{y}^{2}+6{x_3}{y}^{2}+6{x_1}y+16{x_2}y+4y+1 \right)+\cdots. 
\end{multline}

\subsection{The function $\AAA{132}{\underline{\{123,213,231,321\}}}{t,y,x_1,x_2,x_3,x_4}$}
We have an explicit formula for $\AAA{132}{\underline{\{123,213,231,321\}}}{t,y,x_1,x_2,x_3,x_4}$ stated as follows.
\begin{thm}\label{thm:8}The generating function
	\begin{multline}
	\AAA{132}{\underline{\{123,213,231,321\}}}{t,y,x_1,x_2,x_3,x_4}\\
	=(2 x_1^2 x_3 x_4^2 t^4 y^2+2 x_1^2 x_3 x_4^2 t^3 y^2-2
	x_1^2 x_3 x_4 t^4 y^2-2 x_1^2 x_3 x_4 t^3 y-2 x_1^2
	x_3 x_4 t^2 y+x_1^2 x_3 t^3 y-x_1^2 x_4 t^3 y+x_1^2
	t^2\\
	-2 x_1 x_2 x_3^2 x_4 t^4 y^2-2 x_1 x_2 x_3^2
	x_4 t^3 y^2+x_1 x_2 x_3^2 t^4 y^2+x_1 x_2 x_3
	x_4 t^4 y^2+(x_1 t-x_3 t-1) (x_3 t y-x_4 t y+1)\\
	\cdot
	\sqrt{(x_1 t+t y ((x_2-1) x_3 t-x_4)+1)^2-4 t (x_1 (-x_4)
		t y+x_1+x_2 x_3 t y)}+x_1 x_2 x_3 t^3 y-2 x_1
	x_3^2 x_4 t^3 y^2\\
	+x_1 x_3^2 t^4 y^2+x_1 x_3^2 t^3 y+2
	x_1 x_3^2 t^2 y-2 x_1 x_3 x_4^2 t^4 y^2-2 x_1 x_3
	x_4^2 t^3 y^2-2 x_1 x_3 x_4^2 t^2 y^2+x_1 x_3 x_4
	t^4 y^2+3 x_1 x_3 x_4 t^3 y^2\\
	+3 x_1 x_3 x_4 t^3 y+2
	x_1 x_3 x_4 t^2 y+2 x_1 x_3 x_4 t y-x_1 x_3
	t^3 y-2 x_1 x_3 t^2 y-x_1 x_3 t^2-x_1 x_4^2 t^3 y^2+3
	x_1 x_4 t^2 y-2 x_1 t\\
	+x_2 x_3^3 t^4 y^2+2 x_2 x_3^3
	t^3 y^2+x_2 x_3^2 x_4 t^4 y^2+2 x_2 x_3^2 x_4 t^3 y^2+2
	x_2 x_3^2 x_4 t^2 y^2-2 x_2 x_3^2 t^4 y^2-x_2 x_3^2
	t^3 y^2\\
	-x_2 x_3^2 t^3 y-x_2 x_3 x_4 t^3 y^2-x_2 x_3
	t^2 y-x_3^3 t^4 y^2+x_3^2 x_4 t^4 y^2-x_3^2 x_4 t^3
	y^2-x_3^2 t^3 y^2-x_3^2 t^3 y+x_3^2 t^2 y+x_3 x_4^2 t^3
	y^2\\
	+x_3 x_4 t^3 y^2-x_3 x_4 t^2 y^2-2 x_3 x_4 t^2
	y-x_3 t^2 y+x_3 t y+x_3 t+x_4^2 t^2 y^2-2 x_4 t y+1)/(2
	x_3 t y (-x_1 x_4 t+x_3 t+x_4) \\
	\cdot(x_1 (-x_4) t
	y+x_1+x_2 x_3 t y)).
	\end{multline}
\end{thm}

\begin{proof}
We have enumerated the consecutive patterns 123, 213 and 231 in $\Snn(132)$ separately in Section 3 which will help in our computation of $\AAA{132}{\underline{\{123,213,231,321\}}}{t,y,x_1,x_2,x_3,x_4}$.
To compute the generating function tracking the four consecutive patterns of length $3$ simultaneously, we need to count the number of $321$-matches directly from Dyck paths. By \lref{p1}, it is not hard to see that the number of consecutive  $321$-matches is equal to the number of interior segments of size $1$ in the corresponding Dyck path, which makes the computation possible.
	
	Since we need to consider whether a horizontal segment of a Dyck path is interior or not, we shall define generating functions $A_0$ and $A_1$ similar to those of Section 3.3. Both $A_0$ and $A_1$ track the same statistics as $A$. $A_0$ is a generating function of Dyck paths while we take every segment as an interior segment, and $A_1$ is a generating function of Dyck paths while we take the last segment as an interior segment. 
	
	We shall begin with computing a recursion for $A_0$ using the first Dyck path recursion. If the last segment of a path is of size $k$, then there are $k$ Dyck path structures $P_1,\ldots,P_k$ above the last segment. Each of the paths $P_1,\ldots,P_{k-1}$ has a weight of $A_0$, and the path $P_k$ has a weight of $(x_2(A_0-1)+1)$ since there is a consecutive pattern $213$ between $P_k$ and the last segment when $P_k$ is not empty. Based on our analysis of pattern $123,213,231$ and $321$, we have the weight of the last segment and obtain the following recursion.
	\begin{eqnarray}\label{1232A0}
	A_0&=&1+tx_4yA_0+t^2x_3yA_0(x_2(A_0-1)+1)+t^3x_1 x_3yA_0^2(x_2(A_0-1)+1)+\cdots\nonumber\\
	&=&1+tx_4yA_0+\frac{t^2x_3yA_0(x_2(A_0-1)+1)}{1-tx_1A_0}.
	\end{eqnarray}
	
	Next, we want to compute a recursion for $A_1$. We use the third Dyck path recursion which expands the last horizontal segment before the first return that is of size $k$. Similar to Section 4.1, the contribution of the case when $k=1$ is still $tyA_0$. 
	
	When $k\geq 2$, the path ends with steps $DR^k$ before the first return, and there are $k-1$ Dyck path structures $P_1,\ldots,P_{k-1}$ (counting from top to bottom) above the last segment before the first return. There is also a Dyck path structure after the first return with a weight of $A_0$. Among the paths $P_1,\ldots,P_{k-1}$, we suppose that $P_i$ is the first nonempty path for $i=1,\ldots,k$ ($i=k$ means all of $P_1,\ldots,P_{k-1}$ are empty). If $i=k-1$, then $P_{k-1}$ is the only nonempty path among $P_1,\ldots,P_{k-1}$, and it has weight $x_2(A_1-1)$ since it contains the first segment and there is a consecutive pattern $213$ between $P_{k-1}$ and the last segment. When $i<k-1$, $P_{k-1}$ does not contain the first segment and it has a weight of $x_2(A_0-1)+1$; $P_i$ has a weight of $(A_1-1)$; each of $P_{i+1},\ldots,P_{k-2}$ has a weight of $A_0$. Summing over the cases for $i$ from $1$ to $k$, one can obtain that the total weight of the structures above the last horizontal segment before the first return is 
	\begin{equation*}
	(A_0^{k-3}+A_0^{k-4}+\cdots+1)(x_2(A_0-1)+1)(A_1-1)+x_2(A_1-1)+1.
	\end{equation*}
	Thus by adding all the cases, we can obtain the following recursion for $A_1$:
	\begin{eqnarray}\label{1232A1}
	A_1&=&1+tyA_0+t^2x_3yA_0(x_2(A_1-1)+1)\nonumber\\
	&&+t^3x_1x_3yA_0((x_2(A_0-1)+1)(A_1-1)+x_2(A_1-1)+1)\nonumber\\
	&&+t^4x_1^2 x_3yA_0((A_0+1)(x_2(A_0-1)+1)(A_1-1)+x_2(A_1-1)+1)+\cdots\nonumber\\
	&=&1+tyA_0+t^2x_3yA_0(x_2(A_1-1)+1)+\nonumber\\&& \frac{t^3x_1x_3yA_0(A_1-x_2A_0+x_2A_0A_1-tx_1A_0+tx_1x_2A_0-tx_1x_2A_0A_1)}{(1-tx_1)(1-tx_1A_0)}.
	\end{eqnarray}
	
	Finally for $A$, we expand the last horizontal segment like $A_0$ and analyze the weight of Dyck paths above the last segment like $A_1$ to get
	\begin{eqnarray}\label{1232A}
	A&=&1+tA_1+t^2((x_2(A_0-1)+1)(A_1-1)+x_2(A_1-1)+1)\nonumber\\
	&&+t^3x_1 \left((A_0+1)(x_2(A_0-1)+1)(A_1-1)+x_2(A_1-1)+1\right)	+\cdots\nonumber\\
	&=&1+tA_1 + \frac{A_1-1-tyA_0-t^2x_3yA_0(x_2(A_1-1)+1)}{tx_1x_3yA_0}.
	\end{eqnarray}
	One can solve the quadratic equation (\ref{1232A0}) to get a formula for $A_0$, and then solve $A_1$ from (\ref{1232A1}) and solve $A$ from (\ref{1232A}) to prove \tref{8}.
\end{proof}

\tref{8} gives the generating function with complete information about length $2$ or $3$ consecutive pattern matches in $\Snn(132)$. The generating function $\AAA{132}{\underline{\{123,213,231,321\}}}{t,y,x_1,x_2,x_3,x_4}$ is a root of a quadratic equation, and one can compute its Taylor series expansion using Mathematica as follows. 
\begin{multline}
\AAA{132}{\underline{\{123,213,231,321\}}}{t,y,x_1,x_2,x_3,x_4}=1+t+t^2(y+1)+t^3(x_1+x_2 y+x_3 y+x_4 y^2+y) \\
+t^4(x_1^2+x_1 x_2 y+x_1 x_3 y+2 x_1 y+x_2 x_3 y^2+x_2 x_3 y+2 x_2 x_4 y^2+x_3 x_4 y^2+x_3 y^2+x_3 y+x_4^2 y^3+x_4 y^2 ) \\
+ t^5(x_1^3+x_1^2 x_2 y+x_1^2 x_3 y+3 x_1^2 y+x_1 x_2
x_3 y^2+2 x_1 x_2 x_3 y+3 x_1 x_2 x_4 y^2+x_1
x_3 x_4 y^2+2 x_1 x_3 y^2\\+3 x_1 x_3 y
+3 x_1
x_4 y^2+x_2^2 x_3 y^2+x_2 x_3^2 y^2+3 x_2 x_3
x_4 y^3+2 x_2 x_3 x_4 y^2+3 x_2 x_3 y^2+3 x_2\\
x_4^2 y^3+x_3^2 y^2+x_3 x_4^2 y^3+2 x_3 x_4 y^3+x_3
x_4 y^2+x_3 y^2+x_4^3 y^4+x_4^2 y^3 )+\cdots.
\end{multline}

\section{General results about consecutive pattern distribution in $\Snn(123)$}
In this section, we shall study the distribution of $2$ kinds of consecutive patterns, $1 m (m-1)\cdots2$ and $2 m (m-1)\cdots31$,  in $\Snn(123)$ for any positive integer $m$. We will only track the length of the permutation and the number of consecutive pattern matches, and no longer track the number of descents. That is, for two permutations $\la\mbox{ and }\ga$, we will study functions of the form
\begin{equation}
\BBxt{\la}{\ga}(t,x):=\sum_{n\geq 0}t^n\sum_{\sg\in\Snn(\la)}x^{\ga\text{-mch}(\sg)}.
\end{equation}
By the action reverse-complement, the distribution of consecutive patterns $1 m (m-1)\cdots2$ and $2 m (m-1)\cdots31$ in $\Snn(123)$ are equal to the distribution of consecutive patterns $(m-1)\cdots21m$ and $m (m-2)\cdots21(m-1)$, thus we have
\begin{eqnarray}
\BBxt{123}{1 m (m-1)\cdots2}(t,x)&=&\BBxt{123}{(m-1)\cdots21m}(t,x),\\
\BBxt{123}{2 m (m-1)\cdots31}(t,x)&=&\BBxt{123}{m (m-2)\cdots21(m-1)}(t,x).
\end{eqnarray}
We will compute recursions for the two generating functions above. $\BBxt{\la}{\ga}(t,x)$ is abbreviated as $B$ in each computation in Sections 5 and 6.

\subsection{The function $\BBxt{123}{1 m (m-1)\cdots2}(t,x)$}
\begin{thm}The function $\BBxt{123}{1 m (m-1)\cdots2}(t,x)$ satisfies the recursion
	\begin{equation}\label{123long1}
	B=\frac{1+(x-1)t^mB^m}{1-tB},
	\end{equation}and
	\begin{equation}\label{123longcoeff1}
	\BBxt{123}{1 m (m-1)\cdots2}(t,x)|_{t^nx^k}=\frac{1}{k} \sum_{i=k}^{\lfloor\frac{n}{m}\rfloor} (-1)^{i-k} \binom{2n-mi}{n-mi,n+1-i,k-1,i-k}.
	\end{equation}
\end{thm}
\begin{proof}
	Referring to \lref{p2}, the number of consecutive $1 m (m-1)\cdots2$-matches in $\sg\in\Snn(123)$ is equal to the number of path pattern $DR^m$-matches in the corresponding Dyck path $\Psi(\sg)$. We shall compute the distribution of pattern $DR^m$-matches  in Dyck paths to get the generating function $B$. Note that the number of $DR^m$-matches is equal to the number of horizontal segments of length at least $m$.
	
	We use the first Dyck path recursion by expanding the last horizontal segment. Let $k$ denote the size of the last horizontal segment, then there are $k$ Dyck path structures above the last segment, each with a weight of $B$. The weight of the last horizontal segment is $t^k$ when $k<m$ and $xt^k$ when $k\geq m$ since there is a path pattern $DR^m$ appearing in the last segment when the size of the last segment is at least $m$.
	Thus we have the following equation which proves equation (\ref{123long1}).
	\begin{eqnarray}
	B&=&1+tB+\cdots+t^{m-1}B^{m-1} +x(t^mB^m +\cdots)\nonumber\\
	&=&\frac{1-t^mB^m}{1-tB}+x\frac{t^mB^m}{1-tB}\nonumber\\
	&=&\frac{1+(x-1)t^mB^m}{1-tB}.
	\end{eqnarray}
	
	We can multiply both sides of equation (\ref{123long1}) with $t$ and substitute $tB$ with $F$ to obtain
	\begin{equation}
	F=t \frac{1+(x-1)F^m}{1-F}.
	\end{equation}
	It follows from the Lagrange Inversion Theorem that
	\begin{equation}
	F|_{t^n}=\frac{1}{n}\delta^n(z)\big|_{z^{n-1}},
	\end{equation}
	where $\displaystyle\delta(z)=\frac{1+(x-1)z^m}{1-z}$. Thus,
	\begin{eqnarray}\label{123long111}
	B|_{t^{n-1}x^k}&=& F|_{t^nx^k} \nonumber\\
	&=& \frac{1}{n} (1+(x-1)z^m)^n (1-z)^{-n} \big|_{z^{n-1}x^k}\nonumber\\
	&=& \frac{1}{n} \left(\sum_{i=0}^{n}\binom{n}{i}(x-1)^i z^{mi}\right) \left(\sum_{i\geq 0} \binom{n+i-1}{n-1}z^i\right) \Bigg|_{z^{n-1}x^k} \nonumber\\
	&=& \frac{1}{n} \sum_{i=0}^{\lfloor\frac{n-1}{m}\rfloor} \binom{n}{i}(x-1)^i \binom{2n-mi-2}{n-1} \bigg|_{x^k}\nonumber\\
	&=&\frac{1}{n} \sum_{i=k}^{\lfloor\frac{n-1}{m}\rfloor} (-1)^{i-k} \binom{n}{n-i,k,i-k}\binom{2n-mi-2}{n-1}\nonumber\\
	&=&\frac{1}{k} \sum_{i=k}^{\lfloor\frac{n-1}{m}\rfloor} (-1)^{i-k} \binom{2n-mi-2}{n-mi-1,n-i,k-1,i-k},
	\end{eqnarray}
	and equation (\ref{123longcoeff1}) follows by substituting $n$ with $n+1$ in equation (\ref{123long111}).
\end{proof}

Let $m=3$, we obtain a formula for the coefficient of $t^nx^k$ in $\AAAover{123}{132}(t,1,x)$ which proves equation (\ref{thm1eq}) in \tref{1}.

\subsection{The function $\BBxt{123}{2 m (m-1)\cdots31}(t,x)$}
\begin{thm}The function $\BBxt{123}{2 m (m-1)\cdots31}(t,x)$ satisfies the following recursion:
	\begin{equation}\label{123long2}
	B=t \frac{B^{m+2}+(x-1)(B-1)^{m-1}}{B^{m-1}(B-1)},
	\end{equation}and
	\begin{equation}\label{123longcoeff2}
	\BBxt{123}{2 m (m-1)\cdots31}(t,x)|_{t^nx^k}= \frac{1}{n} \sum_{i=0}^{\lfloor\frac{mn-1}{m+2}\rfloor}(-1)^{mn+n+k+1}\binom{n}{i} \binom{mn-mi-2n+i}{mn-1} \binom{mn-mi-n+i}{k},
	\end{equation}
	here $\displaystyle \binom{\alpha}{k}=\frac{\alpha(\alpha-1)\cdots(\alpha-k+1)}{k!}$ is the generalized binomial coefficient.
\end{thm}
\begin{proof}
	Referring to \lref{p2}, the number of consecutive $2 m (m-1)\cdots31$-matches in $\sg\in\Snn(123)$ is equal to the number of $DR^{m-1}D$-matches in the corresponding Dyck path $\Psi(\sg)$. We shall compute the distribution of $DR^{m-1}D$-matches in Dyck paths to get the generating function $B$.
	
	In fact, the number of $DR^{m-1}D$-matches is equal to the number of horizontal segments of length $m$ which are not the last segment, and we need to define a generating function $B_1$ that tracks the same statistics as $B$ while we suppose the last segment is an interior segment.
	
	We use the first Dyck path recursion for both $B_1$ and $B$. First for $B_1$, we let $k$ denote the size of the last horizontal segment, then there are $k$ Dyck path structures above the last segment, each with a weight of $B_1$. The weight of the last horizontal segment is $t^k$ when $k\neq m-1$ and $xt^k$ when $k= m-1$ since it is considered to be interior (i.e.\ we suppose there is a $D$ step following the path) and there is a path pattern $DR^{m-1}D$ appearing in the last segment.
	Thus we have
	\begin{equation}\label{123long2B1}
	B_1=1+tB_1+\cdots+t^{m-2}B_1^{m-2} +xt^{m-1}B_1^{m-1} +t^{m}B_1^{m}+\cdots=\frac{1}{1-tB_1}+(x-1)t^{m-1}B_1^{m-1}.
	\end{equation}
	Then for $B$, we expand the last segment which is of size $k$. There are $k$ Dyck path structures above the last segment, each with a weight of $B_1$. The weight of the last horizontal segment is $t^k$ for any $k$ since it does not contain the pattern $DR^{m-1}D$. It follows that
	\begin{equation}\label{123long2B}
	B=1+tB_1+t^2B_1^2+\cdots=\frac{1}{1-tB_1}.
	\end{equation}
	Equation (\ref{123long2B}) implies that $B_1=\frac{B-1}{tB}$. Substituting $B_1$ with $\frac{B-1}{tB}$ in (\ref{123long2B1}) and multiplying both sides with $\frac{tB^2}{B-1}$ gives (\ref{123long2}).
	
	Next, we shall compute the coefficient using the Lagrange Inversion Theorem. It follows that
	\begin{equation}
	B|_{t^n}=\frac{1}{n}\delta^n(z)\big|_{z^{n-1}},
	\end{equation}
	where $\displaystyle\delta(z)=\frac{z^{m+2}+(x-1)(z-1)^{m-1}}{z^{m-1}(z-1)}$. Thus,
	\begin{eqnarray}
	B|_{t^nx^k}&=& \frac{1}{n} (z^{m+2}+(x-1)(z-1)^{m-1})^n (z^{m-1}(z-1))^{-n} \big|_{z^{n-1}x^k}\nonumber\\
	&=& \frac{1}{n} \sum_{i=0}^{\lfloor\frac{mn-1}{m+2}\rfloor}\binom{n}{i}(x-1)^{(m-1)(n-i)}z^{(m+2)i} (z-1)^{(m-1)(n-i)-n} \bigg|_{z^{mn-1}x^k} \nonumber\\
	&=& \frac{1}{n} \sum_{i=0}^{\lfloor\frac{mn-1}{m+2}\rfloor}\binom{n}{i}(x-1)^{(m-1)(n-i)} (z-1)^{(m-1)(n-i)-n} \bigg|_{z^{mn-1-mi-2i}x^k} \nonumber\\
	&=& \frac{1}{n} \sum_{i=0}^{\lfloor\frac{mn-1}{m+2}\rfloor}(-1)^{mn+n+k+1}\binom{n}{i} \binom{mn-mi-2n+i}{mn-1} \binom{mn-mi-n+i}{k}.\ \ \qedhere
	\end{eqnarray}
\end{proof}

Let $m=3$, we obtain a formula for the coefficient of $t^nx^k$ in $\AAAover{123}{231}(t,1,x)$ which proves equation (\ref{thm2eq}) in \tref{2}.

\section{General results about consecutive pattern distribution in $\Snn(132)$} 

Given a permutation $\sg=\sg_1\cdots\sg_n\in\Snn(132)$, we let $\Phi'(\sg)$ be the Dyck path obtained by removing the initial $n+1-\sg_1$ $D$ steps from $\Phi(\sg)$, and let $\Phi''(\sg)$ be the Dyck path obtained by removing the last one $R$ step of $\Phi'(\sg)$. For any Dyck path $P$ beginning with $D^r R$, we also let $P'$ denote the path obtained by removing the initial $r$ D steps from $P$. 

For example, the Dyck path corresponds to the permutation $\sg=42351\in\Sn{5}(132)$ is $\Phi(\sg)=DDRDDRRRDR$, then $\Phi'(\sg)=RDDRRRDR$, $\Phi''(\sg)=RDDRRRD$ and $(DDRDDRRRDR)'=RDDRRRDR$. Then according to the bijection of Krattenthaler \cite{Kr},
we have the following theorem.
\begin{thm}\label{thm:general}
	Let $m$ be a positive integer and let $\ga=\ga_1\cdots\ga_m\in\Sn{m}(132)$, then
	\begin{enumerate}[(a)]
		\item if $\ga_m=m$, then the distribution of consecutive $\ga$-matches in $\Snn(132)$ is equal to the distribution of $\Phi'(\ga)$-matches in $\Dn$;
		
		\item if $\ga_{m-1}\ga_m=m1$, then the distribution of consecutive $\ga$-matches in $\Snn(132)$ is equal to the distribution of $\Phi''(\ga)$-matches in $\Dn$.
	\end{enumerate}
\end{thm}
The proof of \tref{general} is completely based on the construction of the map $\Phi$. We shall give a sketch of this proof. 
\begin{proof}
We let $\sg=\sg_1\cdots\sg_n$ be any 132-avoiding permutation and let $P=\Phi(\sg)$. Let $Q_1,Q_2$ be two sub-paths consisting of consecutive steps of $P$ begin and end with $R$ steps, and let $\tau_1$, $\tau_2$ be the sub-permutation of $\sg$ above the paths $Q_1$ and $Q_2$ respectively. Then according to the construction of the map $\Phi$, $\red(\tau_1)=\red(\tau_2)$ if $Q_1=Q_2$, i.e.\ a sub-path determines the reduction of the sub-permutation above it, and one can count consecutive patterns from sub-paths.

On the other hand, each consecutive pattern $\ga$ corresponds to several path patterns. In fact, let $\Phi(\ga)=P_1\cdots P_k$ be the Dyck path correspond to $\ga$ where each $P_i$ is a non-empty path with no return position, then each path pattern corresponds to $\ga$ is of the form $P_1' D^{r_1} P_2 D^{r_2}\cdots P_{k-1} D^{r_{k-1}} P_k$, where each $r_i$ is a non-negative integer.

For part (a), if $\ga_m=m$, then $\Phi(\ga)$ has no return position, and there is only one path pattern $\Phi(\ga)'=\Phi'(\ga)$ corresponds to $\ga$. For part (b), if $\ga_{m-1}\ga_m=m1$, then $\Phi(\ga)$ has only one return position before the \thn{n} column, and each path pattern corresponds to $\ga$ is of the form $\Phi'(\ga_1\cdots\ga_{m-1}) D^r R$ for some $r\geq 1$. In other words, each path pattern corresponds to $\ga$ is of the form $\Phi'(\ga_1\cdots\ga_{m-1}) D=\Phi''(\ga)$.
\end{proof}

Based on \tref{general}, we can prove several general results about consecutive pattern distributions in $\Snn(132)$.

\subsection{The function $\BBxt{132}{1\cdots m}(t,x)$}
\begin{thm}The function $\BBxt{132}{1\cdots m}(t,x)$ satisfies the following recursion:
	\begin{equation}\label{132long}
	B=\frac{1-t^{m-1}B^{m-1}}{1-tB}+\frac{t^{m-1}B^{m-1}}{1-txB},
	\end{equation}and
	\begin{equation}\label{132longcoeff}
	\BBxt{132}{1\cdots m}(t,x)|_{t^nx^k}= \frac{1}{n+1} \sum_{i=0}^{\lfloor\frac{n}{m-1}\rfloor} \sum_{j=0}^{n+1-i} (-1)^{mj-j} \binom{n+1}{i} \binom{i+k-1}{k} 
	\binom{2n-mi-mj+j-k}{n-i} .
	\end{equation}
\end{thm}
\begin{proof}
	Referring to \tref{general}(a), the number of consecutive  $1\cdots m$-matches in $\sg\in\Snn(132)$ is equal to the number of path pattern $\Phi'(1\cdots m)=R^m$-matches in the corresponding Dyck path $\Phi(\sg)$. A horizontal segment of size $k\geq m$ contains $k-m+1$ pattern $R^m$-matches, and we can expand the last segment of a Dyck path to get a recursion for $B$.
	
	Let $k$ denote the size of the last horizontal segment, then there are $k$ Dyck path structures above the last segment, each with a weight of $B$. The weight of the last horizontal segment is $t^k$ when $k\leq m-1$ and $x^{k-m+1}t^k$ when $k\geq m$. Thus we have
	\begin{equation}\label{132longB}
	B=1+tB+\cdots+t^{m-1}B^{m-1} + t^{m}xB^{m} +t^{m+1}x^2B^{m+1}+\cdots=\frac{1-t^{m-1}B^{m-1}}{1-tB}+\frac{t^{m-1}B^{m-1}}{1-txB}.
	\end{equation}
	
	We can multiply both sides of equation (\ref{132longB}) with $t$  and substitute $tB$ with $F$ to obtain
	\begin{equation}
	F=t\left( \frac{1-F^{m-1}}{1-F}+\frac{F^{m-1}}{1-xF}\right).
	\end{equation}
	It follows from the Lagrange Inversion Theorem that
	\begin{equation}
	F|_{t^n}=\frac{1}{n}\delta^n(z)\big|_{z^{n-1}},
	\end{equation}
	where $\displaystyle\delta(z)=\frac{1-z^{m-1}}{1-z}+\frac{z^{m-1}}{1-xz}$. Thus,
	\begin{eqnarray}
	B|_{t^{n-1}x^k}&=& F|_{t^nx^k} \nonumber\\
	&=& \frac{1}{n} \left(\frac{1-z^{m-1}}{1-z}+\frac{z^{m-1}}{1-xz}\right)^n\Bigg|_{z^{n-1}x^k} \nonumber\\
	&=& \frac{1}{n} \sum_{i=0}^{\lfloor\frac{n-1}{m-1}\rfloor} \binom{n}{i} z^{(m-1)i} \left(\sum_{a\geq 0}\binom{i+a-1}{a}x^az^a \right) \left( \sum_{j=0}^{n-i}(-z)^{(m-1)j}\right) \nonumber\\
	&&\cdot\left(\sum_{b\geq 0} \binom{n-i+b-1}{b} z^b\right)\Bigg|_{z^{n-1}x^k}\nonumber\\
	&=& \frac{1}{n} \sum_{i=0}^{\lfloor\frac{n-1}{m-1}\rfloor} \binom{n}{i} \binom{i+k-1}{k} z^{mi-i+k} 
	\left( \sum_{j=0}^{n-i}(-z)^{(m-1)j}\right) 
	\left(\sum_{b\geq 0} \binom{n-i+b-1}{b} z^b\right)\Bigg|_{z^{n-1}}\nonumber\\
	&=& \frac{1}{n} \sum_{i=0}^{\lfloor\frac{n-1}{m-1}\rfloor} \sum_{j=0}^{n-i} (-1)^{mj-j} \binom{n}{i} \binom{i+k-1}{k} 
	\binom{2n-mi-mj+j-k-2}{n-i-1},
	\end{eqnarray}
	and equation (\ref{132longcoeff}) follows by substituting $n$ with $n+1$.
\end{proof}

Setting $m=3$, we obtain the formula for the coefficient of $t^nx^k$ in $\AAAover{132}{123}(t,1,x)$ which proves equation (\ref{thm4eq}) in \tref{4}.

\subsection{The function $\BBxt{132}{a12\cdots(a-1)(a+1)\cdots m}(t,x)$}
\begin{thm}\label{general132}
	The function $\BBxt{132}{a12\cdots(a-1)(a+1)\cdots m}(t,x)$ satisfies the following recursion:
	\begin{equation}\label{132general}
	B=\frac{1+t^{m-1}B^{m-a}(x-1)(B-1)}{1-tB}.
	\end{equation}
\end{thm}

\begin{proof}
	By \tref{general}(a), the number of consecutive
	pattern $\ga=a12\cdots(a-1)(a+1)\cdots m$-matches in $\sg\in\Snn(132)$ is equal to the number of path pattern $\Phi'(\ga)=RD^{a-1}R^{m-1}$-matches in the corresponding Dyck path $\Phi(\sg)$. We shall expand the last segment of a Dyck path $P$ using the first Dyck path recursion. Suppose the size of the last segment is $k$, then there are $k$ Dyck path structures above the last segment each with a weight of $B$. 
	
	The weight of the last segment is $t^k$, except when the last segment is part of a pattern $RD^{a-1}R^{m-1}$, which means that the Dyck path $P$ ends with $RD^{a-1}R^{m-1+i}$ for some $i\geq 0$. $\ga_1=a$ means that the path is at the \thn{m-a+i} diagonal and there are $m+1-a+i$ Dyck path structures, $P_1,\ldots,P_{m+1-a+i}$, before the second step of the last $RD^{a-1}R^{m-1+i}$-match, in which $P_{m+1-a+i}$ must be nonempty. Thus the total  weight of the $m+1-a+i$ Dyck path structures is $B^{m-a+i}(B-1)$. The weight of the path $RD^{a-1}R^{m-1+i}$ from the second step is $t^{m-1+i}x$, and the case when the last segment is part of a pattern $RD^{a-1}R^{m-1+i}$ gives a correction weight of
	\begin{equation*}
	t^{m-1}B^{m-a}(x-1)(B-1)+t^{m}B^{m-a+1}(x-1)(B-1)+\cdots=\frac{t^{m-1}B^{m-a}(x-1)(B-1)}{1-tB}
	\end{equation*}
	to the generating function $B$, and the recursion for $B$ is
	\begin{equation}
	B=\frac{1}{1-tB}+\frac{t^{m-1}B^{m-a}(x-1)(B-1)}{1-tB}=\frac{1+t^{m-1}B^{m-a}(x-1)(B-1)}{1-tB}.\qedhere
	\end{equation}
\end{proof}

\subsection{The function $\BBxt{132}{\ga}(t,x)$ when $\ga_1=m-1$ and $\ga_m=m$}
\begin{thm}\label{thm:132new1}
	For any $\ga\in\Sn{m}(132)$ with $\ga_1=m-1$ and $\ga_m=m$, the function  $\BBxt{132}{\ga}(t,x)$ satisfies the following recursion:
	\begin{equation}
	B=\frac{1+t^{m-1}(x-1)(B^2-B)}{1-tB}.
	\end{equation}
\end{thm}
\begin{proof}
	By \tref{general}(a), the distribution of such a consecutive pattern $\ga$ in $\Snn(132)$ is equal to the distribution of $\Phi'(\ga)$ in $\Dn$. In this case, $\Phi'(\ga)=RP_0R$, where $P_0=\Phi(\ga_2\cdots\ga_{m-1})$ is a Dyck path of size $m-2$.
	
	We use the first Dyck path recursion for $B$. By expanding the last segment of a Dyck path $P$ when the last segment is of  size $k$, there are $k$ Dyck path structures above the last segment each with a weight of $B$. The weight of the last segment is $t^k$, except when the last segment is part of a pattern $RP_0R$. Similar to Section 6.2, the correction weight is
	\begin{equation*}
	t^{m-1}B(x-1)(B-1)+t^{m}B^2(x-1)(B-1)+\cdots=\frac{t^{m-1}(x-1)(B^2-B)}{1-tB},
	\end{equation*}
	and the recursion for $B$ is
	\begin{equation}
	B=\frac{1}{1-tB}+\frac{t^{m-1}(x-1)(B^2-B)}{1-tB},
	\end{equation}
	which proves the theorem.
\end{proof}

\subsection{The function $\BBxt{132}{2\cdots m1}(t,x)$}
\begin{thm}The function $\BBxt{132}{2\cdots m1}(t,x)$ satisfies the following recursion:
	\begin{equation}\label{132long1}
	B=t\frac{B^{m-1}+(x-1)(B-1)^{m-1}}{B^{m-4}(B-1)},
	\end{equation}and
	\begin{equation}\label{132longcoeff1}
	\BBxt{132}{2\cdots m1}(t,x)|_{t^nx^k}=\frac{1}{n} \binom{n}{k} \sum_{i=0}^{n-k} (-1)^{mk+k+i+n+1}\binom{mi-i-n}{n+1-mk-k}.
	\end{equation}
\end{thm}
\begin{proof}
	Referring to \tref{general}(b), the number of consecutive  $\ga=2\cdots m1$-matches in $\sg\in\Snn(132)$ is equal to the number of path pattern $\Phi''(\ga)=R^{m-1}D$-matches in the corresponding Dyck path $\Phi(\sg)$, which is equal to the number of horizontal segments of size $k\geq m-1$ that are not the last segment. We shall define a generating function $B_1$ where we track the same statistics as $B$ and we suppose the last segment is an interior segment.
	
	We first derive the recursion for $B_1$. Let $k$ denote the size of the last horizontal segment, then there are $k$ Dyck path structures above the last segment, each with a weight of $B_1$. The weight of the last horizontal segment is $t^k$ when $k< m-1$ and $xt^k$ when $k\geq m-1$ since it is considered to be interior. Thus we have
	\begin{equation}\label{132long1B1}
	B_1=1+tB_1+\cdots+t^{m-2}B_1^{m-2} + t^{m-1}xB_1^{m-1} +t^{m}xB_1^{m}+\cdots=\frac{1+(x-1)t^{m-1}B_1^{m-1}}{1-tB_1}.
	\end{equation}
	For the function $B$, we expand the last segment which is of size $k$. There are $k$ Dyck path structures above the last segment, each with a weight of $B_1$. The weight of the last horizontal segment is $t^k$ for any $k$ since it does not contain the pattern $R^{m-1}D$. It follows that
	\begin{equation}\label{132long1B}
	B=1+tB_1+\cdots=\frac{1}{1-tB_1}.
	\end{equation}
	Equation (\ref{132long1B}) implies that $B_1=\frac{B-1}{tB}$. Substituting $B_1$ with $\frac{B-1}{tB}$ in (\ref{132long1B1}) and multiplying both sides with $\frac{tB^2}{B-1}$ gives (\ref{132long1}).

	Next, we shall compute the coefficient using the Lagrange Inversion Theorem. It follows that
	\begin{equation}
	B|_{t^n}=\frac{1}{n}\delta^n(z)\big|_{z^{n-1}},
	\end{equation}
	where $\displaystyle\delta(z)=\frac{z^{m-1}+(x-1)(z-1)^{m-1}}{z^{m-4}(z-1)}$. Thus,
	\begin{eqnarray}
	B|_{t^nx^k}&=& \frac{1}{n} (z^{m-1}+(x-1)(z-1)^{m-1})^n (z^{m-4}(z-1))^{-n}\big|_{z^{n-1}x^k}\nonumber\\
	&=& \frac{1}{n} ((z-1)^{m-1}x+(z^{m-1}-(z-1)^{m-1}))^n (z-1)^{-n}\big|_{z^{mn-3n-1}x^k}\nonumber\\
	&=& \frac{1}{n} \binom{n}{k} (z^{m-1}-(z-1)^{m-1})^{n-k} (z-1)^{-n}\bigg|_{z^{mn-3n-1}}\nonumber\\
	&=& \frac{1}{n} \binom{n}{k} \sum_{i=0}^{n-k} (-1)^i(z-1)^{mi-i-n}z^{(m-1)(n-k-i)}\bigg|_{z^{mn-3n-1}}\nonumber\\
	&=& \frac{1}{n} \binom{n}{k} \sum_{i=0}^{n-k} (-1)^{mk+k+i+n+1}\binom{mi-i-n}{n+1-mk-k}.\qedhere
	\end{eqnarray}
\end{proof}

Setting $m=3$, we obtain the formula for the coefficient of $t^nx^k$ in $\AAAover{132}{231}(t,1,x)$ which proves equation (\ref{thm5eq}) in \tref{5}.

\subsection{The function $\BBxt{132}{a23\cdots (a-1)(a+1)\cdots m1}(t,x)$}

\begin{thm}\label{general2132}
	The function $\BBxt{132}{a23\cdots (a-1)(a+1)\cdots m1}(t,x)$ satisfies the following recursion:
	\begin{equation}\label{132general1}
	B=\frac{tB^{m-a+2}+t^{a-2}(x-1)(B-1-tB)(B-1)^{m-a}}{(B-1)B^{m-a-1}}.
	\end{equation}
\end{thm}
\begin{proof}
	By \tref{general}(b), the number of consecutive
	pattern $\ga=a23\cdots (a-1)(a+1)\cdots m1$-matches in $\sg\in\Snn(132)$ is equal to the number of path pattern $\Phi''(\ga)=RD^{a-2}R^{m-2}D$-matches in the corresponding Dyck path $\Phi(\sg)$. We define $B_1$ to be a generating function of Dyck paths tracking the same statistics as $B$ and we suppose the last segment is an interior segment, then we can compute a recursion for $B_1$ by expanding the last segment. Suppose the size of the last segment is $k$, then there are $k$ Dyck path structures above the last segment each with a weight of $B$.
	
	The weight of the last segment is $t^k$, except when the last segment is part of a pattern $\Phi''(\ga)=RD^{a-2}R^{m-2}D$, which means that the Dyck path $P$ ends with $RD^{a-2}R^{m-2}$, and the extended path $PD$ contains $RD^{a-2}R^{m-2}D$. $\ga_1=a$ means that the path is at the \thn{m-a} diagonal and there are $m+1-a$ Dyck path structures, $P_1,\ldots,P_{m+1-a}$, before the second step of the last $\Phi''(\ga)$-match, in which $P_{m+1-a}$ must be nonempty. Thus the total weight of the $m+1-a$ Dyck path structures is $B_1^{m-a}(B_1-1)$. The weight of the path $RD^{a-2}R^{m-2}$ from the second step is $t^{m-2}x$, and the case when the last segment is part of a pattern $RD^{a-2}R^{m-2}D$ gives a correction weight of
	\begin{equation*}
	t^{m-2}B_1^{m-a}(x-1)(B_1-1)
	\end{equation*}
	to the generating function $B_1$, and the recursion for $B_1$ is
	\begin{equation}\label{long132B1}
	B_1=\frac{1}{1-tB_1}+t^{m-2}B_1^{m-a}(x-1)(B_1-1).
	\end{equation}
	For the function $B$, since the last horizontal segment of a Dyck path does not contribute to pattern $RD^{a-2}R^{m-2}D$-matches, we expand the last segment of a Dyck path to obtain
	\begin{equation}\label{long132B}
	B=\frac{1}{1-tB_1}.
	\end{equation}
	One can prove equation (\ref{132general1}) from equation (\ref{long132B1}) by the substitution in equation  (\ref{long132B}).
\end{proof}

\subsection{The function $\BBxt{132}{\ga}(t,x)$ when $\ga_1=m-1$, $\ga_{m-1}=m$ and $\ga_m=1$}
\begin{thm}
	For any $\ga\in\Sn{m}(132)$ with $\ga_1=m-1$, $\ga_{m-1}=m$ and $\ga_m=1$, the function $\BBxt{132}{\ga}(t,x)$ satisfies the following recursion:
	\begin{equation}
	B=\frac{tB^{3}}{B-1}+t^{m-3}(x-1)(B-1-tB).
	\end{equation}
\end{thm}
\begin{proof}
	This theorem is analogous to \tref{132new1}. By \tref{general}(b), the distribution of such a consecutive pattern $\ga$ in $\Snn(132)$ is equal to the distribution of $\Phi''(\ga)$ in $\Dn$. In this case, $\Phi''(\ga)=RP_0RD$, where $P_0=\Phi(\ga_2\cdots\ga_{m-2})$ is a Dyck path of size $m-3$.	
	
	We define $B_1$ to be a generating function of Dyck paths tracking the same statistics as $B$ and we suppose the last segment is an interior segment. Using the first Dyck path recursion, there are $k$ Dyck path structures above the last segment of size $k$ each with a weight of $B$.
	
	The weight of the last segment is $t^k$, except when the last segment is part of a pattern $\Phi''(\ga)=RP_0RD$. Similar to Section 6.5, the correction weight is $t^{m-2}(x-1)(B_1^2-B_1),$
	which implies that
	\begin{equation}\label{long2132B1}
	B_1=\frac{1}{1-tB_1}+t^{m-2}(x-1)(B_1^2-B_1).
	\end{equation}
	Similarly, we have $\displaystyle B_1=\frac{B-1}{tB}$, and the theorem can be proved by this substitution.
\end{proof}

\section*{Acknowledgment}

The first author and the second author would like to thank the last author, Professor Jeff Remmel who passed away recently, for his mentorship and collaboration.

The authors would also like to thank Quang Bach and Brendon Rhoades for helpful discussions.

\end{document}